\documentclass[11pt, a4paper, reqno]{amsart}

\usepackage{pdflscape} 
\usepackage{hhline} 
\usepackage{array} 

\usepackage{calrsfs}
\DeclareMathAlphabet{\pazocal}{OMS}{zplm}{m}{n}

\usepackage{latexsym,amsmath,amsfonts,amscd,amssymb, mathrsfs, amsthm}
\usepackage[english]{babel}
\usepackage{appendix}
\usepackage{hyperref}
\usepackage{url}
\usepackage{mathtools}
\usepackage{booktabs}

\usepackage{float}

\advance\oddsidemargin by -1cm
\advance\evensidemargin by -1cm
\advance\topmargin by -1.5cm 
\theoremstyle{plain}  
\newtheorem{theorem}{Theorem}[section]

\newtheorem*{theorem*}{Theorem}

\newtheorem{corollary}[theorem]{Corollary}

\newtheorem{proposition}[theorem]{Proposition}

\theoremstyle{definition}
\newtheorem{definition}[theorem]{Definition}

\newtheorem{remark}[theorem]{Remark}

\newtheorem*{claim*}{Claim}

\numberwithin{equation}{section}

\newcommand{\R}{\mathbb{R}}

\newcommand{\g}{\mathfrak{g}}
\newcommand{\h}{\mathfrak{h}}

\newcommand{\n}{\mathfrak{n}}
\newcommand{\Lb}{\pazocal{L}}
\newcommand{\LL}{\mathfrak{L}}
\newcommand{\I}{\mathfrak{I}}
\newcommand{\NN}{\mathfrak{N}}

\newcommand{\heis}{\mathfrak{heis}}

\newcommand{\iso}{\mathrm{iso}}

\newcommand{\ad}{\mathrm{ad}}
\newcommand{\Ad}{\mathrm{Ad}}

\newcommand{\SO}{\mathrm{SO}}
\newcommand{\Aut}{\mathrm{Aut}}

\newcommand{\spann}{\mathrm{span}}

\usepackage{xcolor}

\newcommand{\be}{\begin{equation}}
	\newcommand{\ee}{\end{equation}}

\newcommand{\ben}{\begin{enumerate}}
	\newcommand{\een}{\end{enumerate}}
\newcommand{\bit}{\begin{itemize}}
	\newcommand{\eit}{\end{itemize}}
\newcommand{\edoc}{\end{document}}

\def\br#1\er{{#1}} 
\def\bw#1\ew{\textcolor{brown}{#1}} 
\def\bb#1\eb{\textcolor{blue}{#1}} 
\def\br#1\er{\textcolor{red}{#1}} 
\def\bm#1\em{\textcolor{magenta}{#1}}
\def\bv#1\ev{\textcolor{olive}{#1}}

\makeatletter
\renewcommand{\tocsection}[3]{%
	\indentlabel{\@ifnotempty{#2}{\ignorespaces#1 #2\quad}}#3}
\renewcommand{\tocsubsection}[3]{%
	\indentlabel{\@ifnotempty{#2}{\ignorespaces#1 #2\quad}}#3}

\newcommand\@dotsep{4.5}
\def\@tocline#1#2#3#4#5#6#7{\relax
	\ifnum #1>\c@tocdepth 
	\else
	\par \addpenalty\@secpenalty\addvspace{#2}%
	\begingroup \hyphenpenalty\@M
	\@ifempty{#4}{%
		\@tempdima\csname r@tocindent\number#1\endcsname\relax
	}{%
		\@tempdima#4\relax
	}%
	\parindent\z@ \leftskip#3\relax \advance\leftskip\@tempdima\relax
	\rightskip\@pnumwidth plus1em \parfillskip-\@pnumwidth
	#5\leavevmode\hskip-\@tempdima{#6}\nobreak
	\leaders\hbox{$\m@th\mkern \@dotsep mu\hbox{.}\mkern \@dotsep mu$}\hfill
	\nobreak
	\hbox to\@pnumwidth{\@tocpagenum{\ifnum#1=1\fi#7}}\par
	\nobreak
	\endgroup
	\fi}
\AtBeginDocument{%
\expandafter\renewcommand\csname r@tocindent0\endcsname{0pt}
}
\def\l@subsection{\@tocline{2}{0pt}{2.5pc}{5pc}{}}
\makeatother

\usepackage{csquotes}

\begin{document}
\title{Isometries of 3-dimensional semi-Riemannian Lie groups}


\author[S. Chaib]{Salah Chaib}
\address{\hspace{-5mm} Salah Chaib, Centro de Matem\'{a}tica,
	Universidade do Minho,
	Campus de Gualtar,
	4710-057 Braga,
	Portugal} 
\email {salah.chaib@cmat.uminho.pt}
\author[A.C. Ferreira]{Ana Cristina Ferreira}
\address{\hspace{-5mm} Ana Cristina Ferreira, Centro de Matem\'{a}tica,
	Universidade do Minho,
	Campus de Gualtar,
	4710-057 Braga,
	Portugal} 
\email {anaferreira@math.uminho.pt}

\author[A. Zeghib]{Abdelghani Zeghib}
\address{\hspace{-5mm} Abdelghani Zeghib, UMPA, CNRS, \'Ecole Normale Sup\'erieure de Lyon, 46, All\'ee d'Italie 69364 Lyon Cedex 07, France }
\email{abdelghani.zeghib@ens-lyon.fr}



\thanks{Corresponding author: Ana Cristina Ferreira (anaferreira@math.uminho.pt)}

\vspace*{-3mm}

\begin{abstract} Let $G$ be a connected, simply connected three-dimensional Lie group (unimodular or non-unimodular) equipped with a left-invariant (Riemannian or Lorentzian) metric $g$. By definition, the isometry group $\mathrm{Isom}(G, g)$ contains $G$ itself, acting by left translations. It turns out that, generically, $\mathrm{Isom}(G, g)$ is actually equal to $G$, and the natural question then becomes to classify those special metrics for which this is not the case. Using Lie-theoretical methods, we present a unified approach to obtain all pairs $(G, g)$ whose full isometry group $\mathrm{Isom}(G, g)$ has dimension greater than or equal to four. As a consequence, we determine, for every pair $(G, g)$, up to automorphism and scaling, the dimension of $\mathrm{Isom}(G, g)$, which can be three, four, or six.

	\medskip
	
	
\end{abstract}

\maketitle

\vspace*{-5mm}

{\small
\tableofcontents
}

\vspace*{-5mm}

\section{Introduction}

Symmetries are a central topic of study in all of mathematics. In semi-Riemannian geometry, symmetries are manifested by isometries. Given two semi-Riemannian manifolds $(M,g_M)$ and $(N, g_N)$, an isometry is a diffeomorphism $\Phi: M \longrightarrow N$ such that $\Phi^\ast g_N = g_M$, that is, the pullback of $g_N$ via $\Phi$ is $g_M$. Much information on this topic can be found, for instance, in the classical textbook of B. O'Neill  \cite[Ch. 8 and 9]{ONeill}.

The isometry group of a semi-Riemannian manifold has the structure of a Lie group in a natural way (\cite[Ch. 9, Th. 32]{ONeill}). For compact Riemannian manifolds, it is well-known that the isometry group is compact \cite{MyersSteenrod}. For semi-Riemannian manifolds this is not always the case, however, proved to be true for Lorentzian simply connected, compact (real analytic) manifolds,  see \cite{D'Ambra} and examples therein. In a recent publication \cite{ZhuChenLiang}, it was shown that the isometry group of a left-invariant semi-Riemannian metric of a connected, simply connected, compact simple Lie group is compact.

Lie groups equipped with a left-invariant metric, here called metric Lie groups, are of significant interest in the literature. One of their features is that it is immediate from the definitions that left translations are isometries.  It might however happen that there exist ``extra isometries'', that is isometries which are not left translations. The typical example is that of  the affine group of the line,  $\mathrm{Aff}^+(\R)$,  the group of $a x + b$-transformations  ($a>0$). Any left-invariant Riemannian metric on {$\mathrm{Aff}^+(\R)$} is isometric, up to scaling, to that of the hyperbolic plane. Therefore, it has extra isometries since its full isometry group is 3-dimensional and is precisely $\mathrm{PSL}(2, \R)$.

In general, the dimension of the full isometry group is greater or equal to $n=\dim G$ and is necessarily less or equal to the maximally symmetric case which is $n(n+1)/2$. This case is realized by space forms which are connected, simply connected semi-Riemannian manifolds which have constant sectional curvature and are geodesically complete \cite[Ch. 8, Def. 20]{ONeill}.

In particular, for a 3-dimensional metric Lie group, the dimension of the full isometry group is between 3 and 6. It is known that this dimension cannot be 5 (see, for instance, \cite{AlloutBelkacemZeghib}), so the only possibilities are 3, 4, or 6. 

Many articles have been devoted to the study of isometries of left-invariant metrics on Lie groups of dimension 3. In the seminal article of Milnor, \cite{Milnor}, the list of 3-dimensional unimodular Lie algebras was derived, and curvature properties, such as constant sectional curvature, of Riemannian metrics were studied.   In \cite{Shin}, a characterization, which turns out to be independent of curvature properties, of the isometry group of left-invariant Riemannian metrics on each of the six unimodular connected simply connected 3-dimensional Lie groups is given. The same question is considered in \cite{HaLee12}, where a different approach is used relying on normal forms of metrics and principal Ricci curvatures. Only recently, as far as the authors know, there were developments in the Lorentzian or in the non-unimodular Lie group cases. More precisely, a similar approach to that of \cite{HaLee12} was employed in \cite{BoucettaChakkar} to determine the isometry groups of all 3-dimensional, connected, simply connected and unimodular Lie groups endowed with a left-invariant Lorentzian metric. On the other hand, in  \cite{CosgayaReggiani}, the full isometry group of any Riemannian left-invariant metric on a simply connected,
non-unimodular Lie group of dimension 3 was computed. 

The case of left-invariant Lorentzian metrics on non-unimodular (connected, simply connected) Lie groups remained, therefore, unstudied. We remark that there are many more non-unimodular Lie algebras than unimodular ones and, moreover, they come in one-parameter families. Recall that the Lie algebra of the isometry group of a semi-Riemannian manifold is the Lie algebra of complete Killing vector fields. In the Riemannian case, all left-invariant metrics are complete and so the completeness of all Killing vector fields is guaranteed. The situation is quite different in the Lorentzian setting, since geodesic completeness often fails. For instance, incomplete metrics of constant sectional curvature cannot have isometry group of maximal dimension and, therefore, incomplete Killing fields exist. In particular, methods like the version of Singer's theorem, \cite{Singer}, which appears in \cite[Th. 2.1]{CosgayaReggiani} cannot be used. 

This paper is part of a program that aims at completing the picture for 3-dimensional Lie groups in terms of isometries and geodesic completeness. Here, we present an ``all at once'' approach based on Lie theoretical methods that determines, up to automorphism of the Lie algebra and scaling of the metric, all left-invariant Riemannian and Lorentzian metrics on all connected, simply connected Lie groups whose isometry group has dimension greater or equal to four.

\subsection{Bianchi classification}

The classification of 3-dimensional Lie algebras goes back to Bianchi in the late 19th century, \cite{Bianchi}. In Table \ref{Table:LA}, we present the list of all 3-dimensional Lie algebras (up to isomorphism) together with their preferred brackets as well as their automorphism groups.

For each connected and simply connected Lie group $G$, the left-invariant metrics will be given at the Lie algebra level $\g$ using the basis of 1-forms dual to the corresponding preferred basis of Table \ref{Table:LA}. Furthermore, the metrics will be given in normal form, that is, we will present a representative of a given metric in their orbit under the action of $\R^*\times\mathrm{Aut}(\g)$. Observe that by identifying the metrics whose isometry Lie algebra is of dimension $6$, Table \ref{Table:space-forms}, and of dimension $4$, Table \ref{Table:Isot-1dim}, we are automatically  identifying the metrics whose isometry Lie algebra is of dimension 3, as they are the remaining ones. The list of normal forms of metrics for each Lie algebra can be found in Appendix \ref{Ap:normal-forms}.

{\small
\begin{table}[H]
\begin{tabular}{*3l} \toprule  
 Lie algebra &  Non-vanishing brackets & Automorphism group  \\ \toprule
 $\R^3$ &  -- & $\mathrm{GL}(3, \R)$\\ \midrule
 $\mathfrak{so}(3)$ &  $[e_1,e_2]=e_3, [e_2,e_3]=e_1, [e_3,e_1] =e_2$\hspace*{5mm} & $\mathrm{SO}(3)$ \\ \midrule
 $\mathfrak{sl}(2, \R)$ &  $[e_1,e_2]=e_2, [e_3,e_1]=e_3, [e_2,e_3] =e_1$ & $\mathrm{SO}(2,1)$ \\ \midrule
$\mathfrak{heis}$ &  $[e_1, e_2] = e_3$ & $\left\{ \left( \begin{smallmatrix}
     a & b & 0 \\
     c & d & 0 \\
     e & f & ad-bc
 \end{smallmatrix}\right): {\tiny ad-bc\neq 0}\right\}$\\ \midrule
$ \mathfrak{euc}(2)$ &  $[e_1,e_2]=e_3, [e_3,e_1]=e_2$ & $\left\{ \left( \begin{smallmatrix}
     1 & 0 & 0 \\
     a & b & -c \\
     d & c & b
 \end{smallmatrix}\right), \left( \begin{smallmatrix}
     -1 & 0 & 0 \\
     a & -b & c \\
     d & c & b
 \end{smallmatrix}\right): {\tiny b^2+c^2\neq 0}\right\}$\\ \midrule
 $\mathfrak{sol}$ & $[e_1, e_2]=e_2, [e_3, e_1]=e_3$ & $\left\{ \left( \begin{smallmatrix}
     1 & 0 & 0 \\
     a & b & 0 \\
     c & 0 & d
 \end{smallmatrix}\right), \left( \begin{smallmatrix}
     -1 & 0 & 0 \\
     a & 0 & b \\
     c & d & 0
 \end{smallmatrix}\right): {\tiny bd\neq 0}\right\}$\\ \midrule 
 $\mathfrak{aff}(\R)\oplus\ \R$ \hspace*{5mm} & $[e_1,e_2]=e_2$ & $\left\{ \left( \begin{smallmatrix}
     1 & 0 & 0 \\
     a & b & 0 \\
     c & 0 & d
 \end{smallmatrix}\right): {\tiny bd\neq 0}\right\}$\\ \midrule
 $\mathfrak{h}(1)$ & $[e_1,e_2]=e_2, [e_1,e_3]=e_3$ & $\left\{ \left( \begin{smallmatrix}
     1 & 0 & 0 \\
     a & b & c \\
     d & e & f
 \end{smallmatrix}\right): {\tiny bf-ce\neq 0}\right\}$\\ \midrule
   $\mathfrak{psh}$ & $[e_1,e_2]=e_2, [e_1,e_3]=e_2+e_3$ & $\left\{ \left( \begin{smallmatrix}
     1 & 0 & 0 \\
     a & b & c \\
     d & 0 & b
 \end{smallmatrix}\right): {\tiny b\neq 0}\right\}$\\ \midrule
  $\mathfrak{h}(\lambda),\,0 <|\lambda|<1$ & $[e_1,e_2]=e_2, [e_1,e_3]=\lambda e_3$ & $\left\{ \left( \begin{smallmatrix}
     1 & 0 & 0 \\
     a & b & 0 \\
     c & 0 & d
 \end{smallmatrix}\right): {\tiny bd\neq 0}\right\}$\\
  \midrule
  $\mathfrak{e}(\mu),\, \mu> 0$ & $[e_1,e_2]=\mu e_2+ e_3, [e_1,e_3]=\mu e_3-e_2$ & $\left\{ \left( \begin{smallmatrix}
     1 & 0 & 0 \\
     a & b & -c \\
     d & c & b
 \end{smallmatrix}\right): {\tiny b^2+c^2\neq 0}\right\}$\\
\bottomrule 
\end{tabular}   \smallskip
\caption{Lie algebras of dimension 3 and their automorphism groups.}\label{Table:LA}
\end{table}
}

\subsection{Strategy and methods}

Let $(G,g)$ be a 3-dimensional connected, simply connected metric Lie group, that is, a Lie group $G$ equipped with a left-invariant (Riemannian or Lorentzian) metric $g$.  Let $L$ denote the full isometry group of $(G,g)$ and $I$ the isotropy subgroup of $L$ at the identity $e_G \in G$, that is, the subgroup of isometries $\Phi$ such that $\Phi(e_G)=e_G$. Denote also by $\g$, $\mathfrak{L}$, and $\mathfrak{I}$ the Lie algebras of $G$, $L$ and $I$, respectively. We will call $\LL$ the Killing algebra of $(G,g)$.

Our strategy is as follows. We consider a Lie algebra $\LL$ of dimension $4$ which is assumed to have a 3-dimensional Lie subalgebra $\g$ which is transversal to a 1-dimensional subalgebra $\I$, that is, $\I \oplus \g = \LL$ (as vector spaces). We split the analysis into two main cases: $\LL$ solvable and $\LL$ not solvable. If $\LL$ is solvable then its derived subalgebra $\NN=[\LL, \LL]$ is nilpotent. Thus, $\NN$ can be either abelian (of dimensions 1, 2 or 3) or the 3-dimensional Heisenberg Lie algebra $\heis$. In the abelian case, we prove that $\NN$ cannot be isomorphic to $\R^3$ or $\R$. We proceed with the analysis in the cases where $\NN$ is either $\R^2$ or $\mathfrak{heis}$. When $\LL$ is not solvable, it is either $\mathfrak{sl}(2,\R)\oplus \R$ or $\mathfrak{so}(3)\oplus \R$. In this case, if $\g$ is solvable then $\g = \mathfrak{aff}(\R)\oplus \R$ and if $\g$ is not solvable then $\g$ is simple, either $\mathfrak{sl}(2,\R)$ or $\mathfrak{so}(3)$. In the simple Lie algebra case, the Killing form will be an important tool.

Our methods are based on investigating the type of infinitesimal isotropy representation  $\iso: \I \longrightarrow \mathfrak{so}(\LL/\I)$.  The Lie algebra $\mathfrak{so}(\LL/\I)$ is identified with $\mathfrak{so}(3)$ in the Riemannian case and with $\mathfrak{so}(2,1)$ in the Lorentzian case. By type of infinitesimal isotropy, we mean that we take a generator $U$ of $\I$ and consider if $\iso(U)$ is of elliptic, hyperbolic or parabolic type  in $\mathfrak{so}(2,1) \simeq \mathfrak{sl}(2,\R)$.  For $\mathfrak{so}(3)$, the isotropy is always of elliptic type (see Subsec. \ref{subsec:inft-isom} for further details). Once the isotropy representation is determined by analysis of the Lie algebra data, we consider which metrics $m$ are such that $\iso(U)$ is an infinitesimal isometry of $m$.  We then apply the action of the automorphism group of $\g$ and scaling on the space of left-invariant metrics to derive the metric normal form.

\subsection{Transversality vs. transitivity} The linear algebra approach which we described in the previous subsection has some important subtleties. If $I$  is the isotropy subgroup at $e_G$ and we identify $G$ with the  left translations $L(G)$, then the full isometry group $L$ is given by $L=I.G$, but there is no straightforward way of determining the algebraic structure of $L$ from the ones of $I$ and $G$.  In particular, $I.G$ is not necessarily a semidirect product. Furthermore, $G$ is normal in $L$ if and only if the connected component of the identity in $I$ is contained in $\Aut(\g)$, see \cite{Shin} or Subsec. \ref{subsec:isom-aut}. 

Now, suppose we start with the algebraic data $(\LL, \g, \I)$. From the transversality assumption $\g \oplus \I = \LL$, it is only possible to know a priori that $G$ acts on $L/I$ with an open orbit. A crucial point is that the purely algebraic approach can also capture incomplete Killing vector fields which only integrate to local isometries. For instance, in Section \ref{sec:g-solv-L-non-solv}, we present the example of the product of 2-dimensional de Sitter space with $\R$ $,(\mathrm{dS_2}, ds^2)\times (\R, dt^2)$, which satisfies all of our assumptions but it does not have, nonetheless, a transitive action of any 3-dimensional Lie group. In section \ref{transitivity}, we prove that the action of $G$  on $L/I$ is necessarily transitive except in the following two cases: (i) $[\LL,\LL] = \R^2 \not \subset \g$ and (ii) $\g$ is solvable but $\LL$ is reductive. It is interesting to observe that these exceptional cases always appear for $\g = \mathrm{aff}(\R)\oplus \R$. In the cases where transitivity is not guaranteed, then further considerations are required, and the necessary analysis is provided also in Sec. \ref{transitivity}.

\begin{remark}
    There is a more general abstract framework on triplets $(L, G, I)$, where $G$ and $I$ are subgroups of $L$, such that $L = I.G$
 (which means that $G$ acts transitively on $L/I$ or alternatively that $I$ acts transitively on $L/G$) without assuming necessarily these actions to be free (equivalently that Lie subalgebras are not necessarily in direct sum). There are classification results, particularly by Onishchik \cite{Onish-TTTG}, but all this with the strong assumption that $L$ is simple (sometimes semi-simple), and specially  that $G$ and  $I$ are reductive. Our situation here is completely different. 
\end{remark}

\subsection{Main results} The contents of this article can be neatly summarized in Tables \ref{Table:space-forms} and \ref{Table:Isot-1dim}. When we set out to determine which metric Lie groups have 4-dimensional isometry group, we will naturally encounter the space forms, since the 6-dimensional full isometry Lie algebra of a space form will have 4-dimensional subalgebras. We identify them by computing their curvature tensors and observing if they are Einstein (in dimension 3, this is equivalent to constant sectional curvature) and checking if the metric is geodesically complete. For the completeness results, we make use of the work of Broomberg and Medina \cite{BrombergMedina} in the unimodular case and of two other articles in our program \cite{CFZ-partI, CF-partII}.
It is well-known that, for a fixed dimension, signature and sectional curvature, all space forms are isometric \cite[Ch. 8, Prop. 23]{ONeill}. The question of which metric Lie groups realize space forms is a natural one. The flat case has been considered by Boucetta and Lebzioui, \cite{BoucettaLebzioui}. In the 3-dimensional case, Table \ref{Table:space-forms} presents (up to automorphism and scaling) a full answer to this question.

{\small
\begin{table}[H]
\begin{tabular}{llc} \toprule  
$\g$ &  Metric normal forms & Sectional curvature\\ \toprule
 $\R^3$ \qquad & $(e^1)^2+(e^2)^2+(e^3)^2$ & 0\\
 &  $(e^1)^2+(e^2)^2-(e^3)^2$ & 0\\ \midrule
$\mathfrak{so}(3)$ &  $(e^1)^2+(e^2)^2+(e^3)^2$ & + \\ \midrule
$\mathfrak{sl}(2,\R)$  &  $(e^1)^2+2(e^2e^3)$ & -- \\ \midrule
$\mathfrak{heis}$ &  $(e^1)^2+2(e^2e^3)$ & 0 \\ \midrule
$\mathfrak{euc}(2)$ & $(e^1)^2+(e^2)^2+(e^3)^2$ & 0 \\
 & $(e^3)^2+(e^2)^2 - (e^1)^2$ & 0 \\ \midrule 
 $\mathfrak{sol}$ &  $(e^1)^2+2(e^2e^3)$ & 0 \\ \midrule
 $\mathfrak{aff}(\R)\oplus \R$ \hspace*{5mm}& $(e^1)^2-(e^3)^2+2(e^2e^3)$ & -- \\ \midrule
$\mathfrak{h}(1)$  & $(e^1)^2+(e^2)^2 + (e^3)^2$ & -- \\ \midrule
$\mathfrak{e}(\mu), \mu >0 $  & $(e^1)^2+(e^2)^2 + (e^3)^2$ &  -- \\ 
 \bottomrule
\end{tabular} \smallskip
\caption{Metric Lie algebras whose Killing algebra is 6-dimensional.}\label{Table:space-forms}
\end{table}
}

Table \ref{Table:space-forms} shows, in particular, that $G=\mathrm{Aff}^+(\R)\times \R$ is the only connected, simply connected  non-unimodular Lie group which can be equipped with a left-invariant Lorentzian metric $g$ such that $(G,g)$ is a space form. This Lie group $G$ has two other (non-equivalent) metrics of constant negative sectional curvature. One of them has full isometry group of dimension 4 and appears in the discussion below. The other is a remarkable case of a metric of constant sectional curvature whose full isometry group is 3-dimensional, the connected component of the identity consists only of left translations (see Subsec. \ref{Subsec:FR-sec-curv} for more details). 

Table \ref{Table:Isot-1dim} presents all the metric connected and simply connected Lie groups (represented by their Lie algebras) whose isometry group is exactly of dimension 4. We add the isotropy type in each case, as well as if $\g$ is an ideal in $\LL$ or not. Also, we include the derived subalgebra of $\LL$, $[\LL,\LL]$, so that the subsection where each metric appears can be easily identified.   

\smallskip

{\small
\begin{table}[H]
\begin{tabular}{*3l*2c} \toprule 
$\g$ &  Metric normal forms & Isotropy type & [$\LL$, $\LL$]  & $\g$ ideal in $\LL$ \\ \toprule
$\mathfrak{so}(3)$ & $(e^1)^2+(e^2)^2+\alpha (e^3)^2, \alpha \neq 0,1$ & elliptic & $\mathfrak{so}(3)$ & yes\\ \midrule
$\mathfrak{sl}(2, \R)$  & $(e^1)^2+2(e^2e^3)+\alpha(e^2-e^3)^2, \alpha\neq 0,\tfrac{1}{2}$ & elliptic & $\mathfrak{sl}(2,\R)$ & yes\\ 
  & $(e^1)^2+2\alpha(e^2e^3), \alpha\neq 0,1$ & hyperbolic & $\mathfrak{sl}(2,\R)$ & yes\\ 
    &  $(e^1)^2+2(e^2e^3)+\varepsilon (e^2)^2, \varepsilon = \pm 1$  & nilpotent & $\mathfrak{sl}(2,\R)$ & yes\\ \midrule
$\mathfrak{heis}$  & $(e^1)^2+(e^2)^2+\varepsilon(e^3)^2, \varepsilon=\pm 1$ & elliptic & $\mathfrak{heis}$ & yes \\ 
   & $(e^1)^2-(e^2)^2+(e^3)^2$ & hyperbolic & $\mathfrak{heis}$ & yes\\ \midrule
$\mathfrak{sol}$  & $(e^2)^2+2(e^1e^3)$ & nilpotent & $\mathfrak{heis}$ & no \\ \midrule
$\mathfrak{aff}(\R)\oplus \R$  & $(e^1)^2+2(e^2e^3)$ & hyperbolic & $\R^2$ & yes\\ 
& $(e^3)^2+2(e^1e^2)$ & nilpotent & $\R^2$ & no\\
& $(e^2)^2+2(e^1e^3)$ & nilpotent & $\mathfrak{heis}$ & no\\
& $(e^1)^2+(e^2)^2+\varepsilon (e^3)^2, \varepsilon = \pm 1$ & elliptic & $\mathfrak{sl}(2,\R)$ & no \\
& $(e^1)^2+(e^2)^2+2(e^2e^3)+\alpha (e^3)^2, \alpha >1, \alpha<0$ & elliptic & $\mathfrak{sl}(2,\R)$ & no \\
& $(e^1)^2-(e^2)^2+2(e^2e^3)+\alpha (e^3)^2, -1<\alpha<0$ & elliptic & $\mathfrak{sl}(2,\R)$ & no \\
\midrule
$\mathfrak{h}(1)$  & $(e^1)^2+(e^2)^2-(e^3)^2$ & hyperbolic & $\R^2$ & yes\\
  & $(e^3)^2+(e^2)^2-(e^1)^2$ & elliptic & $\R^2$ & yes\\
    & $2(e^1e^2)+(e^3)^2$ & nilpotent & $\R^2$ & yes\\
\midrule 
$\mathfrak{psh}$  & $2 (e^1e^2)+(e^3)^2$ & nilpotent & $\R^2$ & yes \\ \midrule
$\mathfrak{h}(\lambda), 0<|\lambda| <1$  & $(e^1)^2+2(e^2e^3)$ & hyperbolic & $\R^2$ & yes\\
  & $(e^2)^2+2(e^1e^3)$ & nilpotent & $\mathfrak{heis}$ & no\\
  & $(e^3)^2+2(e^1e^2)$ & nilpotent & $\mathfrak{heis}$ & no\\ \midrule
  $\mathfrak{e}(\mu), \mu>0$ &  $(e^3)^2+(e^2)^2-(e^1)^2$ & elliptic & $\R^2$ & yes \\
  \bottomrule
\end{tabular} \smallskip
\caption{Metric Lie algebras whose Killing algebra is 4-dimensional.}\label{Table:Isot-1dim}
\end{table}
}

Some comments about Table \ref{Table:Isot-1dim} are in order. We start by observing that $\mathfrak{euc}(2)$ does not appear in this table, there are no metrics of 1-dimensional isotropy. The unique metrics on $\mathfrak{psh}$ and on each $\mathfrak{e}(\mu)$ with $\mu>0$ are {of constant sectional curvature} (zero on $\mathfrak{psh}$ and positive on $\mathfrak{e}(\mu)$). On $\mathfrak{h}(1)$, it is known that all Lorentzian metrics are incomplete, \cite{VukmirovicSukilovic}. They are all Einstein and apart from the Riemannian (which corresponds to hyperbolic space $\mathbb{H
}^3)$, all have isometry group of dimension 4. It can be seen from Appendix \ref{Ap:normal-forms} that all metric normal forms on $\h(1)$ are present in Tables \ref{Table:space-forms} and \ref{Table:Isot-1dim}. A similar situation happens for the Heisenberg group, the list contains all normal forms of metrics. We have a Lorentzian flat space form, and all the other metrics have isometry group of dimension 4. However, all metrics on $\heis$ are complete, since $\heis$ is 2-step nilpotent \cite{Guediri-2step}.

 \begin{remark} As described, our methods here are essentially Lie theoretical, therefore coordinate free, and one of the main subtleties in this approach is that completeness of Killing vector fields does not necessarily hold for Lorentzian Lie groups. The inverse question can be considered: does completeness of all Killing fields imply completeness of the metric? We provide a counter-example to this statement in Subsec. \ref{Subsec:FR-sec-curv}.  
\end{remark}

 \subsection{Discrete isotropies} 
 
 The case where $\dim I = 0$ means that $\dim L = 3$, and that the isotropy $I $ is a discrete group. 
   In the Riemannian case, it is trivial that $I$ is finite, since 
$I$ is discrete in the compact group $\mathrm{O}(3)$. Full isometry groups of left-invariant Riemannian metrics have been computed in the unimodular case in \cite{HaLee12} and in the non-unimodular case in \cite{CosgayaReggiani}. We now focus on the Lorentzian case.

 Since $L \supset G$, $G$ equals $L^0$, the identity component of $L$, and is in particular normal in $L$. Since $L = I.G$ and $I \cap G = 1$, we deduce that 
 $L$ is a semi-direct product $L  =  I \ltimes G$.  So $I $ acts by automorphisms on $\g$ and preserves the Lorentz product on it. 
 But the group of automorphisms of a Lie algebra is algebraic. Also, the group $I^\prime$ of those automorphisms preserving a scalar product 
 is an algebraic group. Since $G$ is simply connected, $I^\prime$ acts on $G$, and hence is contained in the isotropy group of 
 1 for the given metric $h$ on $G$, that is $I^\prime = I$. Finally, an algebraic group has finitely many connected components, and since $I$ is discrete,  $I$ is finite. Now, a finite group  $I \subset \mathrm{O}(2, 1)$,   as any compact subgroup, is contained in a maximal compact subgroup of $\mathrm{O}(2,1)$ which is conjugate to $\mathrm{O}(2)$. Therefore, if $L$ is the isotropy subgroup of a Lorentzian metric, $I$ is a finite subgroup of $\mathrm{O}(2)$.

 \begin{remark}
 It is surely a natural and interesting question, that we do not consider here, to know which finite subgroups of $\mathrm{O}(2)$ can be realized exactly as the isotropy subgroup 
 of a 3-dimensional Lorentzian group?
 \end{remark} 

 \bigskip

\section{Preliminaries}

Let $(G,g)$ be a 3-dimensional metric Lie group, that is, a Lie group $G$ equipped with a left-invariant (Riemannian or Lorentzian) metric $g$. We will assume that $G$ is connected and simply connected (see \cite[Ch. 9, Prop. 20]{ONeill}).  Let $L$ denote the full isometry group of $(G,g)$ and $I$ the isotropy subgroup of $L$ at the identity $e_G \in G$. Denote also by $\g$, $\mathfrak{L}$, and $\mathfrak{I}$ the Lie algebras of $G$, $L$ and $I$, respectively. We shall be referring to $\mathfrak{L}$ as the Killing algebra of $(G,g)$ and to its derived subalgebra $[\mathfrak{L},\mathfrak{L}]$ as the derived Killing algebra. 

An important tool in what follows will be the isotropy representation $\mathrm{Iso}: I \longrightarrow \mathrm{O}(\LL/\I)$ and, more precisely, the infinitesimal isotropy representation $\iso: \I \longrightarrow \mathfrak{so}(\LL/\I)$ which is the quotient representation on $\LL/\I$  of the infinitesimal adjoint representation of $\LL$ restricted to $\I$. Clearly $\mathfrak{so}(\LL/\I)$ is isomorphic to $\mathfrak{so}(3)$ in the Riemannian case, and to $\mathfrak{so}(2,1)$ in the Lorentzian setting. We shall call the elements in $\mathfrak{so}(\LL/\I)$ infinitesimal isometries and detail some of their properties below.

\subsection{Infinitesimal Lorentzian isometries}\label{subsec:inft-isom}

Let $(V, \langle - , - \rangle)$ be a linear Lorentzian space, that is, a vector space $V$ of dimension $n$ equipped with a scalar product of signature $(n-1,1)$. 
 
\begin{definition}
A linear endomorphism $A: V \longrightarrow V$ is said to be an \emph{infinitesimal Lorentzian isometry} if $A$ is skew-symmetric with respect to $\langle -, - \rangle$, i.e.
\begin{equation*}
    \langle A(u),v \rangle + \langle u, A(v) \rangle =0
\end{equation*}
for all $u,v\in V$.
\end{definition}

In dimension $n=3$, infinitesimal Lorentzian isometries have three types according to their eigenvalues. 

\subsubsection{Hyperbolic type}

$A$ has 3 real eigenvalues not all vanishing. 

It can be easily shown that $A = \mathrm{diag}(0, \lambda, -\lambda)$, with $\lambda \neq 0$, and such that the 0-eigenvector is spacelike. The exponential map of $A$, $\mathrm{exp}(A)= \mathrm{diag}(1,\mathrm{e}^\lambda, \mathrm{e}^{-\lambda})$, is thus a linear isomorphism of $V$ which preserves a spacelike direction $\ell$ and hence its orthogonal complement $\mathcal{P}$ is invariant and timelike plane.    

\subsubsection{Parabolic (or nilpotent) type} The only eigenvalue of $A$ is zero with multiplicity 3. 

Then $A$ is nilpotent. It can be readily shown, from the non-degeneracy of $\langle -, - \rangle$  that the rank of $A$ cannot be 1. This means that $A$ cannot be 2-nilpotent. Therefore there exists $u\in V$ such that $\{A^2(u), A(u), u\}$ is a basis of $V$ and in which the matrix of $A$ is written as $\left(\begin{smallmatrix} 0 & 1 & 0 \\ 0 & 0 & 1\\ 0 & 0 & 0\end{smallmatrix}\right)$. The exponential of $A$, $\mathrm{exp(A)}$, is unipotent and has a fixed lightlike direction $\ell$ (the span of $A^2(u)$) and, therefore, also fixes the orthogonal lightlike plane to $\ell$ (which contains $\ell$).

\subsubsection{Elliptic type} $A$ has one real and two complex conjugate eigenvalues.

In this case, it can be easily shown that the real eigenvalue is zero with timelike eigenvector and that the complex conjugate eigenvalues are in fact imaginary with spacelike eigenvectors.
The exponential of $A$, $\mathrm{exp}(A$) acting as an isomorphism of $V$ preserves a timelike direction $\ell$ and preserves its orthogonal plane $\mathcal{P}$ which is then spacelike. Furthermore, $\mathrm{exp}(A)$ acts as a rotation on $\mathcal{P}$.

\subsection{Infinitesimal Riemannian isometries}
    The Riemannian counterpart is easily described. A Euclidean infinitesimal isometry is a linear endomorphism $A: V \longrightarrow V$ which is skew symmetric with respect to a (fixed) positive definite scalar product. Clearly, all such maps are of elliptic type.

\begin{remark}
From the discussion in the 3-dimensional case above, we conclude, in particular, that the rank of any non-trivial infinitesimal isometry (Riemannian or Lorentzian) is 2 and that its trace is 0.    
\end{remark}

\subsection{Isometric automorphisms}\label{subsec:isom-aut}

Given a Lie group $G$, we have an identification of $G$ in the diffeomorphism group of $G$, $\mathrm{Diff}(G)$, via left translations $L(G)$. The normalizer of $G$ in $\mathrm{Diff}(G)$ is 
$$N(G) = \{\varphi\in\mathrm{Diff}(G): \varphi L_g \varphi^{-1} \in L(G), \text{ for all } g\in G\},$$
where the multiplication is obviously given by composition. 

\begin{proposition}
 The normalizer of $G$ in $\mathrm{Diff}(G)$ is the group of affine transformations of $G$, that is, $N(G) = \mathrm{Aut}(G)\ltimes G$.
\end{proposition}

\begin{proof}
       Let $\varphi\in N(G)$, it is sufficient to suppose that $\varphi(e_G)=e_G$ and show that $\varphi$ is an automorphism (since this can easily be obtained up to composing $\varphi$ by a translation, and translations are trivially group normalizing).

For $g\in G,$ the composition $\varphi L_g\varphi^{-1}$ must be some $L_h$ since $\varphi \in N(G)$.
In fact, $h$ is obtained by the image of $e_G$ by $\varphi L_g\varphi^{-1}$ which is $\varphi(g)$.
Now let us compute $\varphi L_{g_1g_2}\varphi^{-1}$ which is $L_{\varphi(g_1g_2)}$, that is
$$ \varphi L_{g_1}L_{g_2}\varphi^{-1} = \left(\varphi L_{g_1}\varphi^{-1} \right) \left( \varphi L_{g_2}\varphi^{-1}\right) = L_{\varphi(g_1)}L_{\varphi(g_2)} = L_{\varphi(g_1)\varphi(g_2)}.$$
Hence, since $L_{\varphi(g_1g_2)}=L_{\varphi(g_1)\varphi(g_2)}$ we conclude that $\varphi(g_1g_2)=\varphi(g_1)\varphi(g_2)$. Therefore, $\varphi$ is an automorphism of $G$. The converse is clear since, for $\varphi \in \mathrm{Aut}(G)$ and $g\in G$, $\varphi L_g \varphi^{-1} = L_{\varphi(g)}$.  
\end{proof}

 The proposition above immediately yields the following.

\begin{corollary}
Let $(G,g)$ be a metric Lie group. If $I$ is the isotropy group of $(G,g)$ at the identity element of $G$, then the isotropy action of $I$ on $G$ is via automorphisms of $G$ if and only if $I \subset N(G)$. 
\end{corollary}

Similar statements to the above corollary have been established and frequently used in the literature, cf. for instance \cite{Shin}.

At the Lie algebra level, given the Lie algebra $\LL$ of the isometry group and the infinitesimal isotropy $\I$, we have a canonical identification of $\LL/\I$ with $\g$ when $\g$ is a reductive complement of $\I$, which, since $\g$ is a subalgebra of $\LL$, means that $\g$ is an ideal of $\LL$. 

\begin{corollary}\label{cor:isotropy}
    If $\g$ is an ideal in $\LL$, the infinitesimal isotropy $\I \longrightarrow \mathfrak{so}(\LL/\I) \simeq \mathfrak{so}(\g)$ acts by derivations of $\g$. If $I_o$ is the connected component of the identity in $I$ then the isotropy $I_o \longrightarrow \mathrm{SO}(\g)$ acts by automorphisms of $\g$.
\end{corollary}

We note that, by Lie's second theorem, when $G$ is simply connected, the groups $\mathrm{Aut}(G)$ and $\mathrm{Aut}(\g)$ are in one-to-one correspondence. Another interesting point is that, in dimension 3, homogeneous semi-Riemannian manifolds are always reductive when the isotropy is connected, \cite{FelsRenner}. Nevertheless, even in the case where $G$ acts transitively on $L/I$, it is possible to have many reductive complements but none that is a subalgebra, see Subsec.\ref{Subsec:FR-red-comp} for an example.

\bigskip

\section{$L$ solvable with abelian derived Killing algebra}\label{sec:N=R2}

We start with the assumptions that  $\dim L = 4$ and that $L$ is solvable. Since $\LL$ is a solvable Lie algebra then its derived algebra $\mathfrak{N}=[\LL, \LL]$ is nilpotent and of dimension less or equal to 3. Therefore, $\NN$ is necessarily isomorphic to one of the following: $\R^3$, $\R^2$, $\R$ or $\mathfrak{heis}$ (the Heisenberg Lie algebra of dimension 3). 

In this section, we will analyze the case where $\NN$ is abelian. The case where $\NN = \mathfrak{heis}$ will be treated in section \ref{sec:N=heis}. 

\subsection{$\mathfrak{N}$ isomorphic to $\R^3$}

In this case, the Lie algebra $\LL$ is $\R \ltimes_\sigma \R^3$ where $\sigma$ is a representation of the Lie algebra $\R$ on $\R^3$, $\sigma: \R \longrightarrow \mathfrak{gl}(\R^3)$. We consider a basis $\{T,X,Y,Z\}$ such that $\{X,Y,Z\}$ spans $\R^3$ and $T$ is such that $\ad_T|_{\R^3}=\sigma(1)$.

The Lie algebra $\g$ of $G$ is a 3-dimensional Lie algebra that can be identified with a subalgebra of $\LL$.  If $\g = \R^3$, then $G$ is  Minkowski spacetime and therefore a space form. Let us then assume that $\g \neq \R^3$. Thus the intersection of $\g$ with $\R^3$ is necessarily of dimension $2$, let us say, $\g \cap \R^3 = \spann\{X,Y\}$. Also $\g$ contains a vector of the form $T+u$, with $u\in\R^3$, but since $\R^3$ is abelian, we can assume that $T$ is chosen such that $\g=\spann
\{T,X,Y\}$.

Let $U$ be a generator of $\I$. If $U\notin\R^3$, then the isotropy of $U$ is the projection of $\ad_U$ acting on $\R^3$. However, since $\R^3 =[\LL,\LL]$ then $\mathrm{iso}(U)$ would be surjective, which cannot happens since $\mathrm{iso}(U)$ is an infinitesimal Lorentzian isometry.  If $U\in \mathbb{R}^3$, then $(\ad_U)^2=0$ on $\LL$ (and thus on $\LL/\I$). Hence $\mathrm{iso}(U)^2=0$, that is, a 2-step nilpotent infinitesimal Lorentzian isometry, which is not possible from the discussion is Subsection \ref{subsec:inft-isom}.

We conclude that $\NN$ cannot be a 3-dimensional abelian subalgebra of $\LL$.

\subsection{$\NN$ is isomorphic to $\R^2$} We will consider two subcases here: $\NN \subset \g$ and  $\NN\not\subset \g$.

\subsubsection{Case where $\NN\subset \g$}\label{subsec:N=R2-ideal} In this case $\g$ is an ideal of $\LL$ and, thus, we can identify $\LL/\I$ with $\g$. For any $U$ which generates $\I$, we can write $\mathrm{iso}(U): \g \longrightarrow \g$, i.e. $\mathrm{iso}(U)$ is the restriction of $\ad_U$ to $\g$. We choose $T$ with $T\in \LL\backslash \NN$ such that, as a vector space, $\g=\R T\oplus \NN$.  Since $\ad_U$ is a derivation (or equivalently, applying Jacobi identity) we obtain that, for any $V\in \g$ 
\begin{equation}
   \ad_U([T,V]) = [\ad_U T, V] + [T, \ad_U V] \label{eq:jacobi} 
\end{equation}
For $V\in \NN$, we have that $[\ad_U T, V]=0$, since $\ad_U T \in \NN$ and $\NN$ is abelian. Therefore, from \eqref{eq:jacobi}, we can write
\begin{equation*}
(\ad_U\circ \ad_T)(V) = (\ad_T \circ \ad_U)(V)    
\end{equation*}
thus, when restricted to $\NN$, $\ad_U$ and $\ad_T$ commute. 
Since $\iso(U)$ has rank 2 and $\NN$ is the derived subalgebra of $\LL$, then the image of $\iso(U)$ is precisely $\NN$. From the discussion of Subsection \ref{subsec:inft-isom}, we can conclude that, up to scaling, $\iso(U)$ restricted to $\NN$ is conjugate to one of the following forms:
\begin{equation*}
  \begin{pmatrix}
        1 & 0 \\
        0 & -1 
    \end{pmatrix}  \mbox{ (hyperbolic case) }, \quad    \begin{pmatrix}
        0 & 1 \\
        0 & 0 
    \end{pmatrix} \mbox{ (parabolic case) }, \quad  \begin{pmatrix}
        0 & -1 \\
        1 & 0 
    \end{pmatrix}  \mbox{ (elliptic case) }. 
\end{equation*}
Now, since $\ad_T$ commutes with $\ad_U$ and both have $\NN$ has an invariant plane, we can deduce the following about $ad_T|_\NN$ and $\g=\R T\ltimes \R^2$:

\begin{enumerate}
    \item Nilpotent case: $\ad_T$ is given by $\left(\begin{smallmatrix}
        \lambda & \nu \\ 
        0 & \lambda
    \end{smallmatrix}\right)$. If $\lambda,\nu\neq 0$, then $\g=\mathfrak{psh}$. If $\lambda = 0$ and $\nu\neq 0$ then $\g=\mathfrak{heis}$.  If $\lambda \neq 0$ and $\nu$=0, then $\g=\h(1)$. The case $\lambda=\nu =0$ gives $\g=\R^3$ which is Minkowski space and, as before, is excluded from the discussion. 
 \item Hyperbolic case: 
    $\ad_T$ is given by $\left( \begin{smallmatrix}
        \alpha & 0 \\
        0 & \beta 
    \end{smallmatrix} \right)$ and $\g$ is one of the Lie algebras $\mathfrak{h}(\lambda)$, including the unimodular $\mathfrak{sol}$, $\mathfrak{h}=\mathfrak{h}(1)$ and $\mathfrak{aff}(\R)\oplus \R$.
    \item Elliptic case: $ad_T$ is of the form $\left( \begin{smallmatrix}
        a & - b\\
        b & a
    \end{smallmatrix}\right)$ with $a^2+b^2\neq 0$.  If $b=0$, then $\g=\h(1)$. If $b\neq 0$, then $\g$ is one of the Lie algebras $\mathfrak{e}(\mu)$ including (when $a=0$) the unimodular $\mathfrak{euc}(2)$. 
\end{enumerate}

We will now analyze in detail the one-dimensional isotropy metrics that appear in the above discussion.

{\it Nilpotent case:} We have a basis $\{T,X,Y\}$ of $\g$ such that $\iso(U)$ is given by the matrix
\begin{equation*}
A = \begin{pmatrix}
0 & 0 & 0 \\
\alpha & 0 & 1\\
\beta & 0 & 0
\end{pmatrix},  \quad \alpha\in \R, \beta\in\R^\ast.
\end{equation*}
If $q$ is a metric such that $A$ is an infinitesimal isometry of $q$ then
\begin{align*}
0 & =  q(T, A(T)) = \alpha q(T,X)+\beta q(T,Y)  \\ 
0 & = q(T, A(X))+q(A(T),X) = \alpha q(X,X) + \beta q(Y,X)\\
0 & = q(T, A(Y))+q(A(T), Y) = q(T,X)+ \alpha q(X,Y) + \beta q(Y,Y)\\
0 & = q(X, A(X)) = 0\\
0 & = q(X, A(Y))+q(A(X),Y) = q(X,X)\\
0 & = q(Y, A(Y)) = q(Y,X)
\end{align*}
Setting $q(T,T) = a$ and $q(Y,Y)=b$, we can then write
\begin{equation*}
  q=  \begin{pmatrix}
        a & -\beta b & \alpha b\\
        -\beta b & 0 & 0 \\
        \alpha b & 0 & b 
    \end{pmatrix}.
\end{equation*}
Interestingly, for all three Lie algebras, $\mathfrak{heis}, \mathfrak{psh}, \mathfrak{h}(1)$, up to automorphism and scaling, this metric is equivalent to $m$ where the non-vanishing terms are
\begin{equation*}
    m(T,X) = 1, \quad m(Y,Y) = 1
\end{equation*}
and, what is more, $m$ is flat. Therefore, the dimension of the Killing algebra depends on whether $m$ is geodesically complete or not. For $\heis$, $m$ is complete, since all metrics are complete for 2-step nilpotent Lie algebras \cite{Guediri-2step}.  For $\mathfrak{h}(1)$, it is known that all Lorentzian metrics are incomplete, \cite{VukmirovicSukilovic}, then so is $m$. As for $\mathfrak{psh}$ the metric is  also incomplete, \cite{CFZ-partI}. Then,  the dimension of the Killing algebra of $m$ is 6 for the case of $\mathfrak{heis}$ and 4 for the case of $\mathfrak{h}(1)$ and $\mathfrak{psh}$.

\smallskip

{\it Hyperbolic case:}  We have a basis $\{T,X,Y\}$ of $\g$ such that $\iso(U)$ is given by the matrix
\begin{equation*}
A = \begin{pmatrix}
0 & 0 & 0 \\
\alpha & 1 & 0\\
\beta & 0 & -1
\end{pmatrix},  \quad \alpha,\beta\in \R.
\end{equation*}
If $q$ is a metric such that $A$ is an infinitesimal isometry of $q$, then setting $q(T,T)=a$ and $q(X,Y)=b$, we have
\begin{equation*}
  q=  \begin{pmatrix}
        a & -\beta b & \alpha b\\
        -\beta b & 0 & b \\
        \alpha b & b & 0 
    \end{pmatrix}.
\end{equation*}
For all Lie algebras $\mathfrak{sol}, \mathfrak{aff}(\R)\oplus \R, \mathfrak{h}(1), \mathfrak{h}(\lambda)$, the metric $q$ is equivalent, up to automorphism and scaling, to $m$ where
\begin{equation*}
    m(T,T) = 1, \quad m(X,Y) =1
\end{equation*}
are the only non-vanishing terms of $m$. For $\mathfrak{sol}$, the metric $m$ is flat and it is complete, as proved in \cite{BrombergMedina}, thus the dimension of the Killing algebra is 6. As for $\mathfrak{h}(1), \mathfrak{aff}(\R) \oplus \R$ and $\mathfrak{h}(\lambda)$ the metric $m$ is of negative sectional curvature  but in all three cases it is incomplete, \cite{CF-partII}, so indeed the Killing algebra has dimension 4.    
\smallskip

{\it Elliptic case:}  We have a basis $\{T,X,Y\}$ of $\g$ such that $\iso(U)$ is given by the matrix
\begin{equation*}
A = \begin{pmatrix}
0 & 0 & 0 \\
\alpha & 0 & -1\\
\beta & 1 & 0
\end{pmatrix},  \quad \alpha,\beta\in \R.
\end{equation*}
If $q$ is a metric such that $A$ is an infinitesimal isometry of $q$, then setting $q(T,T)=a$ and $q(Y,Y)=b$, we have
\begin{equation*}
  q=  \begin{pmatrix}
        a & \beta b & -\alpha b\\
        \beta b & b & 0 \\
        -\alpha b & 0 & b 
    \end{pmatrix}.
\end{equation*}
For all Lie algebras $\mathfrak{euc}(2), \mathfrak{e}(\mu)$ and $ \mathfrak{h}(1)$, the metric $q$ is equivalent, up to automorphism and scaling, to either $m_R = \mathrm{diag}(1,1,1)$ or $m_L=\mathrm{diag}(-1,1,1)$. The metric $m_R$ is Riemannian and therefore complete. For $\mathfrak{euc}(2)$ and $\mathfrak{e}(\mu)$, $m_R$ is flat, and for $\mathfrak{h}(1)$ is of constant negative sectional curvature. For all these cases, therefore,  the Killing algebra is 6-dimensional. For the Lorentzian metric $m_L$, the situation is as follows. For $\mathfrak{h}(1)$, $m_L$ is of constant negative curvature but incomplete, \cite{VukmirovicSukilovic}, thus the Killing algebra is 4-dimensional.  On $\mathfrak{euc}(2)$, $m_L$ is flat and from \cite{BrombergMedina}, we know that $m_L$ is complete,  thus the Killing algebra has dimension 6.  On $\mathfrak{e}(\mu)$, $m_L$ is of positive sectional curvature but incomplete, thus the Killing algebra has dimension 4.

\subsubsection{Case where $\NN \not\subset \g$}\label{subsec:N-not-g} In this case, $\dim [\g,\g]=1$. Once again, $\g=\mathbb{R}^3$ is excluded then $\dim [\g,\g] \neq 0$ and, also, $\dim [\g,\g]\neq 2$ otherwise $\NN=[\g,\g] \subset \g$. Therefore, $\dim[\g,\g]=1$ and $\g$ must be either $\mathfrak{aff}\oplus \R$ or $\mathfrak{heis}$.

\medskip

(I) {\it Lie algebra $\g=\mathfrak{aff}(\R)\oplus\R$}

Consider a basis $\{X,Y,Z\}$ of $\g$ with a unique non-vanishing bracket $[X,Y]=Y$. Let $T\in \NN$, so that $\{X,Y,Z,T\}$ is a basis of $\LL$. Since $[X,Z]=0$ and $\NN$ is 2-dimensional abelian then $\LL$ is a semi-direct product $\R^2\ltimes \R^2$.

Let $U$ be a generator of $\I$. Then $U\notin \g$ and, also, $U\notin \NN$, otherwise the rank of $\iso(U)$ would be less than 2. We then have  $U= T+aX+bY+cZ$ with $a\neq 0$ or $c\neq 0$. We will show that $\iso(U)$ cannot be of elliptic type. Let $\pi: \LL \longrightarrow \LL/\I$ be the canonical projection. For a given $V\in\LL$ we will write $\overline{V}$ for $\pi(V)$. Then $\{\overline{X},\overline{Y},\overline{Z}\}$ form a basis of $\LL/\I$. 
    Suppose that $[T,X]= \alpha Y + \beta T $ and $[T,Z] = \gamma Y + \delta T$, for some constants $\alpha, \beta, \gamma, \delta \in \R$. In the basis $\{\overline{Y}, \overline{Z}, \overline{X}\}$, $\iso(U)$ is given by the matrix
    \begin{equation*}
        \begin{pmatrix}
            a & \gamma - b\delta & \alpha - b (1+\beta) \\
            0 & -c \delta & -c\beta\\
            0 & -a \delta & -a \beta
        \end{pmatrix}.
    \end{equation*}

   If $\iso(U)$ was elliptic, then $\overline{Y}$ would be an eigenvector with eigenvalue $a=0$. Since the trace must also be zero, then iso(U) would be nilpotent, which is a contradiction. 

   Remark the transitivity is not guaranteed in this case, so we have to proceed with care in our analysis.

\smallskip

{\it Nilpotent isotropy} $(a=0)$: In this case, since $c\neq 0$ then $\delta =0$, and from the Jacobi identity $\beta=-1$. Thus, $\iso(U)$ is written as 
\begin{equation*}
    \begin{pmatrix}
        0 & \gamma  & \alpha\\
        0 & 0 & c \\ 
        0 & 0 &0
    \end{pmatrix},
\end{equation*}
with $\gamma \neq 0$. If $\iso(U)$ is an infinitesimal isometry of a metric $q$ then, setting $q(X,X)=x$, $q(Z,Z)=z$, we have on the basis $\{X,Y,Z\}$ that
\begin{equation*}
q=\begin{pmatrix}
    x & -\frac{c}{\gamma}z & \frac{\alpha}{\gamma}z \\
    -\frac{c}{\gamma}z & 0 & 0 \\
    \frac{\alpha}{\gamma}z & 0 & z
\end{pmatrix}.
\end{equation*}

Using the automorphism group of $\mathfrak{aff}(\R)\oplus \R$ and up to scaling, $q$ can be put in normal form as $m$ such that 
\begin{equation*}
    m(X,Y)=1, \quad m(Z,Z) =1.
\end{equation*}

This metric is flat but it is incomplete, \cite{CF-partII}, therefore the dimension of the isometry Lie algebra cannot be equal to $6$ and so it is either 3 or 4. 

Considering $X'=X-\tfrac{\alpha}{\gamma}Z$, the basis $\{X',Y, Z\}$ spans $\mathfrak{aff}(\R)\oplus\R$ and the non-zero bracket relations in $\LL$ are $[X',Y]=Y$, $[X',T]=T, [T,Z]=\gamma Y$. Therefore $\LL$ can also be written as $\mathfrak{aff}(\R)\ltimes \R^2$. We encounter here the case (5) of Th. 1.1. in \cite{AlloutBelkacemZeghib} which is isometric to half-Minkowski space. Thus, the action of $\mathrm{Aff}^+(\R)\times \R$ on $L/I$ is transitive (see Sec. \ref{transitivity} for more details).   
    
\smallskip

{\it Hyperbolic isotropy} $(a\neq 0)$: In Sec. \ref{transitivity}, it will be proved that $\mathrm{Aff}^+(\R)\times \R$ cannot act transitively on $L/I$, in this case.

\medskip

(II) {\it Lie algebra $\g=\mathfrak{heis}$}

Let $\g=\spann\{X,Y,Z\}$ with $[X,Y]=Z$. Consider another element $T\in \LL$ such that $\NN=\spann\{T,Z\}$.  Since  $\NN$ is abelian,  $[Z,T]=0$. Thus, $Z$ is central in $\LL$. Let $U$ be a generator of the isotropy $\I$. Then, up to scaling, $U=T+aX+bY+cZ$.  Suppose that $[X,T]= \alpha Z + \beta T $ and $[T,Z] = \gamma Z + \delta T$, for some constants $\alpha, \beta, \gamma, \delta \in \R$, with $\beta \neq 0$ or $\delta\neq 0$.  The Jacobi identity for $\LL$ implies that $\alpha\beta-\delta\gamma=0$. In the basis $\{\overline{Z}, \overline{X}, \overline{Y}\}$, $\iso(U)$ is given by the matrix
    \begin{equation*}
        \begin{pmatrix}
            0 & c\beta -\alpha - b & c\delta - \gamma + a \\
            0 & a \beta & a\delta\\
            0 & b \delta & b \delta
        \end{pmatrix}.
    \end{equation*}
 Unsurprisingly, $\overline{Z}$ is a 0-eigenvector of $\iso(U)$. We observe that since the lower $2\times 2$-block matrix has zero determinant and its trace is zero, then $\lambda=0$ is the only eigenvalue of $\iso(U)$. Therefore, $\iso(U)$ is necessarily nilpotent. Consider $V=\delta X - \beta Y$. We have that $V\neq 0$ and $[T,V]=0$. Moreover, $[U,V]=-(b\delta+a\beta) Z =0$. Hence, $V$ is another 0-eigenvector of $\iso(U)$, which cannot happen for a non-trivial Lorentzian infinitesimal isometry. Therefore, this case cannot occur.

\subsection{$\NN$ is isomorphic to $\R$} For every $U\in \I$, $\ad_U$ would be of rank less or equal to 1 and the same conclusion follows for $\mathrm{iso}(U): \LL/\I \longrightarrow \LL/\I$. This cannot happen, since $\I$ is a one-dimensional space of infinitesimal Lorentzian isometries. 

\bigskip

\section{$L$ solvable with Heisenberg derived Killing algebra}\label{sec:N=heis}   

We consider now the case where $\NN$ is isomorphic to the Heisenberg algebra $\heis$. 
In this case, the Lie algebra $\LL$ is an extension of $\mathfrak{heis}$, that is, $\R \ltimes_\sigma \mathfrak{\mathfrak{heis}}$ where $\sigma$ is a representation $\sigma: \R \longrightarrow \mathfrak{der}(\mathfrak{heis})$. It was shown in \cite{AlloutBelkacemZeghib} that the metrics with isometry Lie algebra of this type are  plane waves. We will identify here such plane waves admitting a transitive free action of a 3-dimensional Lie group.

 \subsection{Case where $\g = \mathfrak{heis}$}

In this case $\g$ is an ideal of $\LL$, and by Cor. \ref{cor:isotropy},  for any $U$ generator of $\I$, $\iso(U)$ is an infinitesimal automorphism, or, in other words, a derivation of $\mathfrak{g}$. We fix the standard basis $\{X,Y,Z\}$ such that $[X,Y]=Z$ and recall that
$$\mathfrak{der}(\mathfrak{heis}) = \left\{ \begin{pmatrix}
    a_{1} & a_{2} & 0 \\
    a_{3} & a_{4} & 0 \\
    a_{5} & a_{6} & a_{1} + a_{4} 
\end{pmatrix} \, : \, a_1,\cdots, a_6 \in \R \right\}.$$

Thus, any such derivation $\mathcal{D}$ has $Z$ as an eigenvector. Therefore, if $\mathcal{D}$ is an infinitesimal isometry, $Z$ is in the kernel of $\mathcal{D}$, since the trace of $\mathcal{D}$ is zero. This implies that $\iso(U)$ cannot be of parabolic type. For nilpotent infinitesimal isometries, $Z$ is also in the image and since the rank is 2 then the derived algebra of $\LL$ cannot be $\heis$. Thus, $\iso(U)$ is either elliptic or hyperbolic.

\subsection{Case where $\g \neq \mathfrak{heis}$} Let the Heisenberg Lie algebra $\mathfrak{heis}$ be generated by $\{X,Y,Z\}$ with $[X,Y]=Z$.
We consider a basis $\{T,X,Y,Z\}$ of $\LL$ with $T$ such that $\ad_T|_{\mathfrak{heis}}=\sigma(1)$. In this case, we can consider the isotropy to be $\I = \R X$, since all 1-dimensional subspaces different from $\R Z$ are equivalent under the action of $\mathrm{Aut}(\mathfrak{heis})$. The Lie algebra $\g$ is linearly generated by an element $T+u$, $u\in \mathfrak{heis}$, together with a subalgebra $\mathfrak{n}\subset \mathfrak{heis}$ of dimension 2. Up to automorphism of $\mathfrak{heis}$ preserving $\R X$, we can assume that $\mathfrak{n}=\R Y \oplus \R Z$, and therefore a basis of $\g$ is $\{T+u, Y, Z\}$. Since $u\in \mathfrak{heis}$, $u=T+cX+dY+eZ$, for some scalars $c,d,e\in \R$. We can assume that $d=e=0$, as $T+cX$ and $T+u$ have the same brackets on $\g$. Thus $\g$ has basis $\{T+cX, Y, Z\}$.

Consider the isotropy $\iso(X)$ on $\LL/\R X$. As usual,  we can consider the basis $\{\overline{T+cX}, \overline{Y}, \overline{Z} \}$ of $\LL/I$ given by the projection $\LL \longrightarrow \LL/I$. Given $\alpha, \gamma, \delta \in \R$ such that $[T,X]= \alpha X + \gamma Y + \delta Z$, we have  
\begin{equation*}
    \begin{array}{rl}
        \iso(X)(\overline{T+cX}) & =\, \overline{\ad_X(T)}  =  {-\overline{(\alpha X + \gamma Y + \delta Z)}} = -\gamma\overline{Y}-\delta\overline{Z},  \\
         \iso(X)(\overline{Y}) & =\, \overline{\ad_X(Y)}  =   \overline{Z},\\ 
         \iso(X)(\overline{Z}) & = \, \overline{\ad_X(Z)}  =   0.\\ 
    \end{array}
\end{equation*}
In this basis, $\iso(X)$ is represented by the matrix 
\begin{equation*}
    \begin{pmatrix}
        0 & 0 & 0 \\
        -\gamma & 0 & 0\\
        -\delta & 1 & 0
    \end{pmatrix}.
\end{equation*}
Therefore, $\iso(X)$ must be 3-step nilpotent. If so, then $[T,X]$ has non-trivial $Y$ component, i.e. $\gamma \neq 0$. Since $[\LL,\LL] = \mathfrak{heis}$ and $\g$ is a subalgebra,  then $[\g,\g]\subset \n$. In particular, there exists some $\beta, b \in \R$ such that $[T, Y] = \beta Y + b Z$. In sum, and recalling that $\ad_T|_\heis$ is a derivation of $\heis$, then $\ad_T|_\heis$ is represented by the matrix 
\begin{equation*}
    A= \begin{pmatrix}
        \alpha & 0 & 0 \\
        \gamma & \beta & 0 \\
        \delta & b & \alpha + \beta 
    \end{pmatrix}, \,\, \gamma\neq 0, \alpha \neq 0,  \beta \neq 0.  
\end{equation*}

Notice that $\alpha \neq 0$  and $\beta \neq 0$ are given by the fact that $[\LL, \LL] = \mathfrak{heis}$. Observe, furthermore, that this implies that $\g$ cannot be an ideal of $\LL$. 

\subsection{The Lie algebra $\g$} If $\g \neq \mathfrak{heis}$, then $\g$ is generated by $S=T+cX$, $Y$ and $Z$, with bracket relations given by 
\begin{equation*}
    [S,Y] = \beta Y + (b+c) Z, \quad [S,Z] = (\alpha+\beta)Z, \quad [Y,Z] =0. 
\end{equation*}
Thus $\mathfrak{n}=\R Y \oplus \R Z$ is an abelian subalgebra of $\g$ and $\ad_S$ leaves $\mathfrak{n}$ invariant. Thus if $B={\ad_S}|_\mathfrak{n}$, then
\begin{equation*}
    B = \begin{pmatrix}
     \beta & 0 \\
     b+c & \alpha + \beta 
    \end{pmatrix}.
\end{equation*}
Notice that $B$ has eigenvalues $\beta$ and $\alpha+\beta$, and since $\alpha\neq 0$, they are different, and hence $B$ is diagonalizable. Thus we obtain any semidirect product $\g_B = \R \ltimes_B \R^2$ with $B$ diagonalizable but not proportional to the identity.

If $\alpha + \beta \neq 0$ then, $\g_B$ is isomorphic to one of $\h(\lambda)$, $0<|\lambda|<  1$, or to the unimodular $\mathfrak{sol}$. Notice that, up to replacing $Y$ by $\tilde{Y}=Y-\frac{b}{\alpha+\beta}Z$, we can assume that $b=0$ in the matrix $A$. If $\alpha+\beta =0$, then $\g_B$ is isomorphic to $\mathfrak{aff}(\R)\oplus \R$.

\subsection{The one-dimensional isotropy metrics on $\g$} $\phantom{a}$

 \subsubsection{Case 1: $\g = \mathfrak{heis}$}
 The automorphism group of the Heisenberg Lie algebra 
 acts as $\mathrm{GL}(2,\R)$ on the plane generated by $\{X,Y\}$ and leaves the center spanned by $\{Z\}$ invariant. 
This implies that, up to scaling, there are only four orbits of metrics on $\heis$ (three Lorentzian and one Riemannian) and these are determined the sign of $q(Z,Z)$, see also Appendix \ref{Ap:normal-forms}. 
If $\iso(U)$ is hyperbolic then $Z$ is spacelike, we obtain the metric $q=\mathrm{diag}(-1,1,1)$. If $\iso(U)$ is elliptic then, in the Lorentz case, $Z$ is timelike, and we obtain the metric $q=\mathrm{diag}(1,1,-1)$;  in the Riemannian case, we have $q=\mathrm{diag}(1,1,1)$. It is worth mentioning that none of these metrics have constant sectional curvature.

\subsubsection{Case 2: $\g \neq \mathfrak{heis}$ with  $\alpha+\beta \neq 0$} 
We fix $A$ of the form 
$$A=\begin{pmatrix}  \alpha &0&0\\ \gamma  &  \beta  & 0 \\ \delta &  0 & \alpha + \beta\end{pmatrix},\quad \gamma \neq 0,  \alpha \neq 0, \beta \neq 0, $$
which represents $\ad_T$ restricted to $\heis$.
For $c \in \R$,  the linear subspace generated by $\{T+ c X, Y, Z\}$ is a subalgebra that we denote by $\g^c$ (it depends only on $c$ since $\alpha, \beta$, $\delta$ are fixed).
Observe that all these algebras are isomorphic to $\g_C$ given by $C={\left(\begin{smallmatrix}1&0\\0&\frac{\alpha+\beta}{\beta}\end{smallmatrix}\right)}$.
The algebras $\g^c$ are conjugated in $\Lb$ by means of the one-parameter group determined by $X$. The corresponding groups $G^c$ act freely transitively and isometrically on $L/I$.

 Let us determine the possible left-invariant metrics on $\g^0$. We denote by $q$ the inner products associated to the isotropy $\iso(X)$ on $\Lb/\I$ identified with  $\g^0$.
We write  $S=T+cX$ and  take $q(\overline{S},\overline{S})$  and $q(\overline{Y},\overline{Y})$ to be $a_1$ and $a_2$, respectively. Then, using the fact that  $\iso(X)$ is skew-symmetric with respect to $q$, we get the following
\begin{align*}
q(\overline{Z},\overline{Z})&=q(\overline{Z},\iso(X)\overline{Y})=-q(\iso(X)\overline{Z},\overline{Y})=0,\\
    q(\overline{Y},\overline{Z})&= q(\overline{Y},\iso(X)\overline{Y}) = 0,  \\
q(\overline{S},\overline{Z})&=q(\overline{S},\iso(X)\overline{Y})= -q(\iso(X)\overline{S}, \overline{Z}) = \gamma a_2.
\end{align*}
Finally, since $\overline{Y}= -\frac{1}{\gamma}\iso(X)(\overline{S}+\delta\overline{Y})$, we have 
\begin{align*}
    q(\overline{S}, \overline{Y}) & = \frac{1}{\gamma}q(\iso(X)\overline{S}, \overline{S}+\delta \overline{Y}) = -\frac{1}{\gamma}q(\gamma \overline{Y}+\delta\overline{Z}, \overline{S}+\delta\overline{Y}) = -q(\overline{Y}, \overline{S}) -2\delta a_2,
\end{align*}
and, thus $ q(\overline{S}, \overline{Y}) = \delta a_2$.

Considering the action on the space of metrics by automorphisms of $\Lb/\I$, i.e. automorphisms of $\Lb$ fixing the isotropy $\R X$ (hence, automorphisms of $\g$), we obtain  two families of metrics. Their normal forms are given by either $m_1$ or $m_2$ where
\begin{equation*}
    m_1 = \begin{pmatrix}
        0 & 0 & 1\\
        0 & 1 & 0\\
        1 & 0 & 0
    \end{pmatrix} \quad \text{ and } \quad m_2 =  \begin{pmatrix}
        0 & 1 & 0\\
        1 & 0 & 0\\
        0 & 0 & 1
    \end{pmatrix}.
\end{equation*}
We remark that when computing the orbit of $q$ under the action of the automorphism group of $\mathfrak{h}(\lambda)$, $0<|\lambda|<1$, we have to take into consideration that either $\overline{Y}$ or $\overline{Z}$ can take on the role of $e_3$ as given in Table \ref{Table:LA}. 
An exception is made for the case of the Lie algebra $\mathfrak{sol}$ where these two metrics are, in fact, equivalent by means of the extra component in the automorphism group.

\subsubsection{Case 3: $\g \neq \mathfrak{heis}$ with $\alpha + \beta=0$}

 In this case, $\g$ is $\mathrm{aff}(\R)\oplus\R$. The center of $\g$ is, thus, generated by $Z$, the center of $\mathfrak{heis}$. We remark that $\alpha$ and $\beta$ do not contribute to the components of metric associated to the nilpotent infinitesimal isometry and that the automorphism group of $\mathfrak{aff}(\R)\oplus \R$ is the same as the automorphism group of $\h(\lambda)$, $0<|\lambda|< 1$. Therefore, the metrics in question are equivalent to the metric with normal form $m_1$.

\subsection{Space of plane waves} 
As described in Th. 1.4. of \cite{AlloutBelkacemZeghib}, non-unimodular homogeneous 3-dimensional plane waves are parameterized by a real number which we will denote here by $\sigma$.  By non-unimodular we mean that $\LL$ is a non-unimodular Lie algebra, which is equivalent, in our notation, to having $\mathrm{trace}(\ad_T)\neq 0$, i.e. $\alpha+\beta \neq 0$. Our goal here to obtain a basis $\{Z', X', Y'\}$ of $\heis$ giving a similar parametrization to the one in 
\cite{AlloutBelkacemZeghib}.

Let $\mathcal{D}$ be the derivation of $\heis$ corresponding to $\ad_T$. We consider $\mathcal{D}' = \mathcal{D}+\delta \ad_Y$. Since $\ad_Y$ is an inner derivation then the extensions of $\heis$ given by $\mathcal{D}$ and $\mathcal{D}'$ are isomorphic Lie algebras. 
So we can replace $T$ with $T' = T + \delta Y$. Since $\alpha+\beta \neq 0$ we can rescale $T'$  and consider $\widetilde{T}= \frac{1}{\alpha+\beta}T'$.   Notice that the isotropy element $X$ is fixed, so  the structure of $\LL$ as the isometry Lie algebra of a metric in $\g$  remains unchanged. In the basis $\{Z, X, Y\}$, $\ad_{\widetilde{T}}$ is given by 
\begin{equation*}
    \begin{pmatrix}
        1 & 0 & 0\\
        0 & \alpha & 0\\
        0 & \gamma & 1-\alpha
    \end{pmatrix},\, \gamma\neq 0, \alpha \neq 0,1.
\end{equation*}
We take $X' = X$. For $Y$ we rescale it by some $r\neq0$ and perturb it by $X$, 
hence, $Y\mapsto Y'=sX+rY$ (we can also perturb it by $Z$, but this is enough for our aim).
Since we are only changing the basis by $\Aut(\heis)$, we have no freedom in $Z$, it automatically changes to $Z'= rZ$.
The brackets of $\widetilde{T}$ in the new basis are given by
\begin{align*}
    \ad_{\widetilde{T}}(Z')&=Z',\\
    \ad_{\widetilde{T}}(X')&=\frac{r\alpha-s\gamma}{r}X'+\frac{\gamma}{r} Y',\\
    \ad_{\widetilde{T}}(Y')&=\frac{sr(2\alpha-1)-s^2\gamma}{ r}X'+\frac{r-\alpha r+ s\gamma}{r}Y'.
\end{align*}
Choosing $r=\dfrac{\gamma}{\alpha(\alpha-1)}$ and $s=\dfrac{1}{\alpha -1}$,  we have $\ad_{\widetilde{T}}$ represented by the matrix
$$
D_\sigma=\left(\begin{matrix}  1&0&0\\ 0  &  0  & 1 \\ 0 & \sigma & 1\end{matrix}\right),\ \sigma = \alpha(\alpha-1). 
$$
Hence, for every $\sigma=\sigma(\alpha)\neq0$ the plane wave $P_\sigma$ admits the transitive and free action of the Lie group 
$$G_{\sigma}=\R\ltimes_{\rho_{\sigma}}\R^2;\ \rho_{\sigma}(t)=\exp\left(t\left(\begin{smallmatrix}
    1&0\\0&\frac{1}{1-\alpha}
\end{smallmatrix}\right)\right),\  t\in\R.$$
Remark that $\sigma(\alpha)=\alpha(\alpha-1)>-1/4$, so these are the hyperbolic plane waves of \cite[Th. 1.4]{AlloutBelkacemZeghib}.

\bigskip

\section{$\g$ solvable with non-solvable Killing algebra}\label{sec:g-solv-L-non-solv}

We consider here the case where $\g$ is solvable but $\Lb$ is not solvable.
Then, $\Lb$ is isomorphic to either $\mathfrak{so}(3,\R)\oplus\R$ or $\mathfrak{sl}(2,\R)\oplus\R$. Since $\LL$ has a 3-dimensional solvable subalgebra, $\LL$ is necessarily $\mathfrak{sl}(2,\R)\oplus \R$.  
We have that $\g \cap \mathfrak{sl}(2, \R)$ is isomorphic to $\mathfrak{aff}(\R)$, since the intersection must be a two dimensional subalgebra of $\mathfrak{sl}(2,\R)$ and $\R^2$ is not a possibility. In fact it is not difficult to see that, up to conjugacy, $\g = \mathfrak{aff}(\R) \oplus \R$.

We consider a basis $\{X,Y,Z\} $ of $\mathfrak{sl}(2,\R)$ given by the commutation relations $$[X, Y]=Y, [X, Z]=-Z, \mbox{ and } [Y,Z]=X.$$
The subalgebra given by $\R X\oplus\R Y$ is $\mathfrak{aff}(\R)$, which we will assume be to the $\mathfrak{aff}(\R)$ factor of $\g$.  If we let $T\in\Lb$ be the generator of the $\R$-factor, that is, the center of $\g$, then $\g$ is generated by $\{X,Y,T\}$.

Let $U$ be a generator of the isotropy $\I$ and consider $\iso(U): \LL/\I\longrightarrow \LL/\I$, with $\LL/\I =(\mathfrak{sl}(2,\R)\oplus\R) / \R U$. We  consider two cases depending on whether $U$ belongs or not to $\mathfrak{sl}(2,\R)$.

\subsection{$U \in \mathfrak{sl}(2,\R)$}

Let us write $U=\alpha X+\beta Y+\gamma Z$ for some $\alpha$, $\beta$, and $\gamma\in\R$. 
Suppose that $\iso(U)$ infinitesimally preserves a Lorentz metric, and consider the type, hyperbolic, nilpotent (of nilpotency 3) or elliptic, of $\iso(U)$ as a Lorentz isometry. (Notice that this is the same as the type of $U$ seen as an element of $\mathfrak{sl}(2,\R)$). In the basis $\{\overline{X}, \overline{Y}, \overline{T}\}$ of $\LL/\I$, $\iso(U)$ is given by the matrix
\begin{equation*}
    \begin{pmatrix}
        -\alpha & -\gamma & 0\\
        -2\beta & \alpha & 0 \\
        0 & 0 & 0
    \end{pmatrix}.
\end{equation*}
Thus, the type of $\iso(U)$ is given by the value of $-\alpha^2-2\beta\gamma$.

\subsubsection{Parabolic type: $\alpha^2+2\beta\gamma=0$}
 We immediately see from the matrix above that $U$ cannot be parabolic, which would correspond to $\iso(U)$ being nilpotent of nilpotency 3.
Indeed, $\iso(U)$ as an infinitesimal Lorentz isometry preserves only one direction, which then must be the line generated by $\overline{T}$. 
Thus, the 2-dimensional factor $\mathfrak{sl}(2,\R)/\I$ is also $\iso(U)$-invariant. However, a nilpotent infinitesimal Lorentz isometry does not preserve a supplementary hyperplane to its invariant direction.

\subsubsection{Elliptic type: $\alpha^2+2\beta\gamma<0$}\label{sec:sl2-aff-ellip}

If $U$ is elliptic,  we have the example of the product $\mathbb (\mathbb{H}^2, ds^2) \times (\R, -dt^2)$, where $(\mathbb{H}^2,ds^2)$ is the Riemannian hyperbolic plane. 
The group $\mathrm{Aff}^+(\R) \times \R$ acts transitively on it, but its full isometry group is $\mathrm{PSL}(2, \R) \times \R$.

For $U$ elliptic, the $\R$-factor of $\g$ is timelike and its orthogonal plane is space-like (see the discussion in Subsec. \ref{subsec:inft-isom}).
Its orthogonal plane is then isometric to $(\mathbb{H}^2,ds^2)$ up to scaling since we obtain the quotient of $\mathrm{SL}(2,\R)$  by a one-parameter elliptic subgroup conjugated to $\mathrm{SO}(2)$.   
Overall, for $U$ elliptic, the given example is the unique Lorentzian example up to isometry and scaling. 

Similar considerations show that  $\mathbb (\mathbb{H}^2, ds^2) \times (\R, dt^2)$ is the unique Riemannian example up to isometry and scaling.

\subsubsection{Hyperbolic type: $\alpha^2+2\beta\gamma>0$}

Suppose, now, that $U$ is a hyperbolic element.  Then the $\R$-factor of $\g$ is an invariant direction which is spacelike (again, see the discussion in Subsec. \ref{subsec:inft-isom}). Hence, its orthogonal plane is also invariant and the isometric action on it is conjugate to $\left(\begin{smallmatrix}\exp(t\lambda)&0\\0&\exp(-t\lambda)\end{smallmatrix}\right)$ for $\lambda=\sqrt{\alpha^2+2\beta\gamma}$. This implies that the orthogonal plane must have signature $(1,1)$.  

The quotient space $L/I$ is, up to cover, $\mathrm{SL}(2, \R)/I\times \R$.  Therefore, we obtain the product $(\mathrm{dS}_2, ds^2) \times (\R, dt^2)$, where $(\mathrm{dS}_2, ds^2)$ is the de Sitter plane.
 Observe, however, that the $\mathrm{Aff}^+(\R)$-action on $\mathrm{dS}_2$ has an open orbit but is not transitive, therefore, $\mathrm{Aff}^+(\R)\times \R$ does not act transitively on $\mathrm{dS_2}\times \R$.  We have here the example of $\mathrm{dS}_2\times \R$ which is a homogeneous space but does not have a transitive action of a 3-dimensional Lie group.

\subsection{$U\notin \mathfrak{sl}(2,\R)$}\label{subsec:U-not-sl2} In this case, $U$ has a component in $T$, that is, $U=\alpha X + \beta Y + \gamma Z + \delta T$. Since $T$ is central, we can assume that $\delta = 1$, by replacing $\delta T$ by $ T$. The matrix of $\iso(U)$, $U=\alpha X+\beta Y+\gamma Z +T$, is given by   
\begin{equation*}
    \begin{pmatrix}
        -\alpha & -\gamma & 0\\
        -2\beta & \alpha & 0 \\
        -1 & 0 & 0
    \end{pmatrix}
\end{equation*}
and the type of isotropy is determined by the quantity $-\alpha^2-2\beta\gamma$. 

Transitivity is not guaranteed so further discussion is needed. Since $U\notin \mathfrak{sl}(2,\R)$ then $U$ is transversal to $\mathfrak{sl}(2,\R)$ and, as shall be proved in Sec. \ref{sec:g-non-solv}, $\iso(U)$ is an inner derivation of $\mathfrak{sl}(2,\R)$. This implies that the $\mathrm{SL}(2,\R)$ acts transitively on $L/I$.  It will be shown, in Sec. \ref{transitivity},  that $\mathrm{Aff}^+(\R)\times\R$ acts transitively on $L/I$ if and only if $U$ is elliptic.

Performing the usual computations, we obtain the following normal forms for left-invariant metrics on $\mathrm{Aff}^+(\R)\times \R$ whose isometry Lie algebra is isomorphic to $\LL$ and the isotropy to $\mathfrak{so}(2)$:
 \begin{equation*}
  m_{r} = \begin{pmatrix}
        1 & 0 & 0\\
        0 & 1 & 1 \\
        0 & 1 & r
    \end{pmatrix}, r>1, r<0;\,  n_{s} = \begin{pmatrix}
        1 & 0 & 0\\
        0 & -1 & 1 \\
        0 & 1 & s
    \end{pmatrix}, -1<s<0;\, q = \begin{pmatrix}  1 & 0 & 0\\
        0 & 0 & 1 \\
        0 & 1 & -1 \end{pmatrix}.
\end{equation*}
We have that $q$ is of constant negative sectional curvature and is complete (in fact, it corresponds to the Killing form of $\mathfrak{sl}(2,\R)$), therefore the dimension of its Killing algebra is 6. For $m_r$ and $n_s$, the Killing algebra is 4-dimensional.

\bigskip

\section{$\g$ non-solvable}\label{sec:g-non-solv}

We consider the case where $\g$ is not solvable with isometry Lie algebra $\LL$ of dimension 4. Then $\g$ is either $\mathfrak{sl}(2,\R)$ or $\mathfrak{so}(3)$ with $\LL$  isomorphic to $\g \oplus \R$.  

We start by showing that the possible embeddings of $\g$ as a subalgebra of $\LL$ are by automorphisms of $\g$. Let $\rho: \g \longrightarrow \LL$ be such an embedding, i.e. a homomorphism of Lie algebras with trivial kernel. Observe that $\rho$ cannot have image in the $\R$-factor since $\g$ is simple. This implies that the image of $\g$ by $
\rho$ is the graph of representation of $\g$ in $\R$. But the only such representation is the trivial one and thus the graph is equal to $\g\oplus \{0\}$. Then the claim follows.   

Let $U\in \LL$ be a generator of the isotropy $\I$. Then $U$ is transverse to $\g$ and, since $\g$ is an ideal in $\LL$, the isotropy action must be by automorphisms, by Cor. \ref{cor:isotropy}. In particular, identifying $\LL/I$ with $\g$ via the projection by transversality, $\iso(U): \g \longrightarrow \g$ is a derivation of $\g$. Since $\g$ is simple, $\iso(U)$ is an inner derivation.

\subsection{Case $\g=\mathfrak{sl}(2,\R)$} Our reference for the present discussion is the Killing form $\kappa$. Up to automorphism and scaling it corresponds to the only bi-invariant metric  on $\mathrm{SL}(2,\R)$. The metric $\kappa$ is, moreover,  of negative constant sectional curvature and is complete (its geodesics run on 1-parameter subgroups). Therefore, $(\mathrm{SL}(2,\R), \kappa)$ is a space form and its full isometry group is $\mathrm{O}(2,2)$, \cite[Ch. 9]{ONeill}. The isometry Lie algebra is $6$-dimensional with $\mathfrak{o}(2,2) = \mathfrak{sl}(2,\R) \oplus \mathfrak{sl}(2,\R)$.  Note that no other metrics of constant sectional curvature (or Einstein, since we are in dimension 3) exist on $\mathrm{SL}(2,\R)$, so $\kappa$ is the only case where the dimension of the Killing algebra is 6. 

Recall that the Lie algebra of inner derivations of $\mathfrak{sl}(2,\R)$ is $\mathfrak{so}(2,1) \simeq \mathfrak{sl}(2,\R)$, therefore, we can have all types of infinitesimal isometries.

\subsubsection{Elliptic type} In this case, the isotropy action has a timelike invariant line and acts as a rotation on its spacelike orthogonal complement. Thus, the associated metric $q$ will be the sum of the Killing form plus a multiple of $e\otimes e$ where $e$ is a timelike vector (of $\kappa$).  More concretely, there exists a basis $\{f_1, f_2, f_3\}$ such that $\iso(U)=\ad_{f_1}$ and  $[f_3,f_1]=-f_2, [f_3,f_2] =f_2$. This basis can be chosen such that $[f_1,f_2]=f_3$. The Killing form is written as $2\kappa$ where
\begin{equation*}
    \kappa =  f^1 \odot f^1 + f^2\odot f^2 - f^3 \odot f^3, 
\end{equation*}
where $\{f^1, f^2, f^3 \}$ is the dual basis and $\odot$ denotes the symmetric tensor product.
The metrics compatible with this infinitesimal isometry are given by the matrix $\mathrm{Diag}(a,a,b)$, for some $a,b \in \R^\ast$. Therefore, up to scaling,
\begin{equation*}
    q_\alpha = \kappa + \alpha (f^3\odot f^3), \quad \alpha \neq 0,1,
\end{equation*}
is the one-parameter family of metrics with 4-dimensional isometry Lie algebra of elliptic type. Notice that  Riemannian metrics are included in the analysis.

Converting to our standard basis of Table \ref{Table:LA}, by means of the transformation $e_1=f_1, e_2=\tfrac{f_2+f_3}{\sqrt{2}}, e_3=\tfrac{f_2-f_3}{\sqrt{2}}$, we have
\begin{equation*}
    q_\alpha= e^1 \odot e^1 + 2 e^2 \odot e^3 + \tfrac{\alpha}{2} (e^2 - e^3) \odot (e^2 - e^3), \quad \alpha \neq 0,1.
\end{equation*}

\subsubsection{Hyperbolic type} In this case, the isotropy action has an invariant spacelike direction, and its orthogonal plane is then timelike and intersects the null cone in two lines corresponding to the two non-zero eigenvectors of $\iso(U)$. Thus, the associated metric $q$ will be the sum of the Killing form plus a multiple of $e\otimes e$ where $e$ is a spacelike vector (of $\kappa$).  More concretely, there exists a basis $\{e_1,e_2,e_3\}$ such that $\iso(U)=\ad_{e_1}$ and $[e_1,e_2]=e_2$, $[e_1, e_3]=-e_3$. This basis can be chosen such that $[e_2,e_3]=e_1$. (Notice that the basis here is already standard). The Killing form is written as $2\kappa$ where
$$\kappa =  e^1 \odot e^1 + 2 e^2 \odot e^3.$$ 
The metrics compatible with this infinitesimal isometry are given by the matrix 
\begin{equation*}
    \begin{pmatrix}
        a & 0 & 0 \\
        0 & 0 &  b \\
        0 & b & 0
    \end{pmatrix},
\end{equation*}
for some $a,b \in \R^\ast$. Then 
\begin{equation*}
    q_\beta = \kappa + \beta (e^1 \odot e^1), \quad \beta \neq 0,-1,
\end{equation*}
as the one-parameter family of metrics, up to and scaling, with 4-dimensional isometry Lie algebra of hyperbolic type. 

\subsubsection{Nilpotent type} In this case, we have a unipotent automorphism $A$ having an invariant null line, say, the line generated by $f_1$ and preserving its orthogonal plane spanned by $\{f_1,f_2\}$. Hence, $A$ preserves forms admitting this plane as a kernel. In particular $A$ preserves the dual form $f^3$. It follows that $A$ preserves any quadratic form which is the sum of the Killing form and a multiple of $f^3\otimes f^3$.    More concretely, there exists a basis $\{f_1, f_2, f_3\}$ such that $\iso(U)=\ad_{f_1}$ and  $[f_1,f_2]=f_1, [f_1,f_3] = f_2$. This basis can be chosen such that $[f_2,f_3]=f_3$. The Killing form is written as $2\kappa$ where
\begin{equation*}
    \kappa =  f^2 \odot f^2 - 2 f^1\odot f^3. 
\end{equation*}
The metrics compatible with this infinitesimal isometry are given by the matrix 
\begin{equation*}
    \begin{pmatrix}
        0 & 0 & -a \\
        0 & a &  0 \\
        -a & 0 & b
    \end{pmatrix},
\end{equation*}
for some $a \in \R^\ast, b\in \R$. Then 
\begin{equation*}
    q_\gamma = \kappa + \gamma (f^3 \odot f^3), \quad \gamma \neq 0,
\end{equation*}
is the one-parameter family of metrics, up to scaling, with 4-dimensional isometry Lie algebra of hyperbolic type. 

To convert to our standard basis of Table \ref{Table:LA}, we only have to take $e_1=f_2, e_2=f_3, e_3=-f_1$, and we obtain
$$q_\gamma = e^1\odot e^1 +2 e^2\odot e^3 + \gamma (e^2 \odot e^2), \gamma \neq 0.$$

\subsubsection{Normal forms} As can be seen from Appendix \ref{Ap:normal-forms}, the metrics $q_\alpha$, of elliptic isotropy, are already in normal form. As for the metrics $q_\beta$, of hyperbolic isotropy, we rescale them as $q_{\tilde{\beta}} = e^1 \odot e^1 + 2\tilde{\beta} (e^2\odot e^3)$, $\tilde{\beta}\neq 0,1$, and the $q_{\tilde{\beta}}$ are also in normal form. The metrics $q_\gamma$, of nilpotent isotropy, are not, however in normal form. Let $\gamma >0$, and consider the transformation (in the standard basis) given by
\begin{equation*}
    Q= \begin{pmatrix}
        1 & 0 & 0 \\
        0 & \frac{1}{\sqrt{\gamma}} & 0 \\
        0 & 0 & \sqrt{\gamma}
    \end{pmatrix}.
\end{equation*}
Clearly $Q \in \mathrm{Aut}(\mathfrak{sl}(2,\R))$ and it transforms $q_\gamma$, with $\gamma>0$, into $q_1$. Similarly, if $\gamma <0$,  $q_\gamma$ is seen to be equivalent to $\gamma_{-1}$.

\subsection{Case $\g=\mathfrak{so}(3)$} The isotropy representation is identified with the usual representation of $\SO(3)$ on $\R^3$, in particular, all one-parameter subgroups are conjugated. Thus, the isotropy Lie algebra consists only of elliptic infinitesimal isometries.
We consider the standard basis $\{e_1,e_2,e_3\}$ of $\mathfrak{so}(3)$  with commutator relations $$[e_1,e_2]=e_3,\quad [e_2,e_3]=e_1,\quad [e_3,e_1]=e_2.$$
The Killing form is negative definite and given by $\kappa = -2 (e^1\odot e^1 + e^2\odot e^2 + e^3 \odot e^3)$.

The $\SO(2)$ action fixes a line we can assume is $e_3$ and acts by rotation on its orthogonal plane $\{e_1,e_2\}$.
This gives, up to scaling, a one-parameter family of inner products by the line generated by $\{e_3\}$ (the timelike direction in the Lorentzian case) and the plane $\{e_1,e_2\}$ spacelike.
Namely, we get
$$
q_\delta=(e^1\odot e^1)+(e^2\odot e^2)+\delta(e^3\odot e^3), \ \delta \neq 0,1.
$$

Another way to characterize this family of metrics, is to observe that any inner product is given by means of the Killing form $\kappa$ and a $\kappa$-symmetric endomorphism $B$ as  $\kappa(B\cdot,\cdot)$. Then $B$ is diagonalizable and the inner product has a one-dimensional isotropy if and only if $B$ has two equal eigenvalues, which can be made equal to 1 after rescaling. The isotropy has exactly dimension one since the only metrics with constant sectional curvature are multiples of the Killing form.

\bigskip

\section{Transversality vs Transitivity} \label{transitivity}

 We consider the $G$-action on $L/I$, where $\g$ is transversal to $\I$, that is,  $\g\oplus\I=\LL$. 
 This implies that the $G$-orbit of $1 I \in L/I$ is open, that is $G\cdot I$ (defined as the set of products $g i, g \in G, i \in I$) is open in $L$.
 As for the transitivity of the $G$-action on $L/I$, it means that $G\cdot I = L$. It is worth observing that $G$ acts transitively on $L/I$ is equivalent to that $I$ acts transitively on $L/G$ (as both properties are equivalent to $L = G\cdot I$).

We saw in Sec. \ref{sec:g-solv-L-non-solv} that in general transversality does not imply transitivity:  $\mathrm{Aff}^+(\R) \times \R$ has an open orbit when acting on $\mathrm{dS}_2 \times \R$ but does not act transitively.
 
There is a simpler, in fact, solvable example in dimension 2:  $\mathrm{Aff}^+(\R)$ acts on the Minkowski plane $\R^{1,1}$ (of dimension $1 + 1$) with the upper and lower half-planes as open orbits and the real axis as a singular orbit. 
 
More generally, in order to see how an orbit $G\cdot xI$ becomes singular, left translate it by $x^{-1}$ to get $\Ad_x(G)\cdot I$ whose tangent space is $\Ad_x(\g) + \I$. 
Hence, it has dimension 3, and $Gx$ is open if $\I \not \subset  \Ad_x (\g)$, otherwise, $\I \subset \Ad_x \g$, and $Gx$ has dimension 2. 

\medskip
 
Consider  $X = L/G$, which has dimension 1,   endowed with the $L$-left action.  We would like to know when a given one-parameter group of $ L$ acts transitively on $X$  and now proceed to answering this question.

\begin{proposition}  \label{tansitivity1} Let $N$ be a  connected nilpotent Lie group acting  analytically on  a one dimensional manifold $X$. If there is no common fixed point of $N$, then 
the $N$-action preserves a Riemannian (flat) metric, unique up to a constant, so that $N$ acts by translation, say,  $X$ is identified with one of the abelian groups $\R$ or $\SO(2)$ and $N$ acts via a homomorphism of this group. 
\end{proposition}

\begin{proof} Assume first that $N$ is abelian.  Assume there exists a one-parameter subgroup  $J$ acting freely. This means all orbits are open, and hence by connectedness, there exists 
exactly one orbit. Thus $X$ is identified with $J \sim \R$ acting by translation on itself. But any diffeomorphism of $\R$ commuting with translations is a  translation. Now,  let $n_0 \in N$ having a non-trivial fixed point set: $X \neq Fix(n_0) \neq \emptyset$.  By analyticity, this is a discrete set. By commutativity, $N$ preserves $Fix(n_0)$, and 
by connectedness of $N$, it acts trivially on it, that is, $Fix(n_0) = Fix(N)$, contradicting our hypothesis.

In the general nilpotent case, one takes $J$ and $n_0$ in the center $C$  of $N$. This can be done, unless $C$ acts trivially on $X$. In this last case, one considers the action of 
$N/C$, and apply to it the same process (one argues by induction on the nilpotency degree of $N$).

\end{proof}

Let us first focus on the case where $X = \R$.

\begin{proposition}  \label{transitivity2}
If $L$ has a normal nilpotent subgroup $N$  not contained in $G$, then, 

\begin{enumerate} 

\item 
 $X$ is identified with $\R$, and the $L$-action on $X$  factorizes via a homomorphism
$\rho: L \to  \mathrm{Aff}^+(\R)$.

\item 

Let  $Tr \subset L$ be the inverse image by $\rho$ of the translation group of $\mathrm{Aff}^+(\R)$ (a normal subgroup).  Let also
$G_0$ be the maximal normal subgroup contained in $G$ (This is the subgroup of $G$ acting trivially on 
$L/G$, equivalently, $G = \cap_{l \in L} \Ad_l(G)$). Then

- $Tr = G_0. N$. 

 - At the Lie algebra level, $\mathfrak{Tr} =  \mathfrak{g}_0 + \mathfrak{n}$, where $\g_0 = \{u \in \g / \ad( \LL ) \subset \g \}$.

\item[] A one-parameter group $J \subset L$  acts  transitively or trivially on $X$  
if and only if $J \subset Tr$.

\item

A  one-parameter group $J$  whose $\Ad$-action is parabolic or elliptic  acts trivially or transitively on $X$.   

\end{enumerate}

\end{proposition}

\begin{proof}  $\phantom{a}$

\begin{enumerate}

\item 

$N$ acts by translation on $\R$.  Since $N$ is normal in $L$, any element $l \in L$ acts by sending a translation to a translation.  It follows that 
$l$ acts affinely.

\item 

Straightforward. To see that $\g_0$ is the Lie algebra of $G_0$, observe that $g = \exp u$, acts trivially on $L/G$, exactly, 
its derivative on the tangent space $\Lb/\g$ at the base point $1.G \in L/G$  is trivial. This is equivalent to $\ad_u(\Lb) \subset \g$, that is
$u \in \g_0$.

\item 
Now an affine transformation of $\R$ is either a translation or has a fixed point where it has a non-trivial derivative, say $\lambda >1$, or $0< \lambda <1$. This implies, for the element 
$l \in L$, that its $\Ad$-action has $\lambda$ as eigenvalue, and so cannot be parabolic or elliptic.
\end{enumerate}
\end{proof}

\begin{remark} In the other case of Proposition \ref{tansitivity1}, 
 the 1-dimensional manifold is the circle $\SO(2)$.  Proposition \ref{transitivity2} applies more easily in this case,
however, no affine transformations exist in this instance.
\end{remark}

\subsection{Applications} There are two easy cases where transitivity is guaranteed.

\subsubsection{Case where $L$ is a Heisenberg extension} From Sec. \ref{sec:N=heis}. Proposition  \ref{transitivity2} applies.  If $\g \neq \heis$, we have that the isotropy $I$ is nilpotent, and so we get transitivity.  If $\g = \heis$, then $\g$ is an ideal in {$\LL$}.

\subsubsection{Case $\NN= [\Lb,\Lb]$ is $\R^2$ and $[\LL, \LL] \subset \g$} From Subsec. \ref{subsec:N=R2-ideal}. In this case, $\g$ is an ideal in $\LL$.  

\medskip 

We will now consider the less straightforward cases.

\subsubsection{Case $\NN = [\Lb,\Lb]$ is $\R^2$ and $[\LL, \LL] \not\subset \g$}  From Subsec. \ref{subsec:N-not-g}.  We will see  here that    transitivity holds 
only in the case of nilpotent  isotropy.  So all the other potential  cases  do not have a 4-dimensional isometry group. Here $N$ is $[L, L]$, with Lie algebra $\NN = [\Lb, \Lb] = \R Y \oplus \R T$. The ideal $\g_0$ consists of $\R Y$ together with 
$\R W$ where  $W \in \R X \oplus \R Z$, satisfies  that  $\ad_W(\R Y \oplus \R T) \subset \g$, equivalently, $[W, T]\in \g$. It follows that  $[W, T] \in \R Y$. 
If  transitivity holds, then $\mathfrak{I} $ must be contained in $\mathfrak{Tr} = \R W \oplus \NN$. 

\medskip

- Hyperbolic and elliptic cases.   Let $U = W + V$, $V \in \NN$. If $\ad_U$ induces a hyperbolic infinitesimal isometry on $\Lb/ \g$, then it will has a spectrum
$\{\alpha, - \alpha \neq 0, 0\}$. Since $\ad_U$ sends all $\Lb$ on $\NN$, it cannot have a non-zero eigenvector outside 
$\NN$. Thus, it must be hyperbolic on $\NN$. There, the action of $\ad_U$  equals that of $\ad_W$, which can not be hyperbolic, since 
$\ad_W(T) \in \R Y$.  The case of elliptic isotropy is similar, but it was also ruled out directly in Subsec. \ref{subsec:N-not-g}. 

\medskip

- Nilpotent case. Now $\ad_U$ or equivalently $\ad_W$, acting on $\NN$ must be nilpotent. Conversely, if $P \in \R X \oplus \R Z$, is such that  $\ad_P$ is nilpotent on $\NN$, then like $W$, $\ad_P$  sends
$\NN$ to $\R Y$. Indeed, $Y$ is a common eigenvector for $X$ and $Z$, and hence for $P$ too, and 
by nilpotency, $\ad_P Y = 0$. Again by nilpotency, $\ad_P T \in \R Y$. It now suffices to be sure that $\ad_U $ acting on $\Lb/ \I$ has nilpotency 3 in order to preserve infinitesimally some Lorentz metric, see Subsec. \ref{subsec:N-not-g}

\medskip

\subsubsection{$\g$ solvable and $\Lb$  not solvable} Here  $L= \widetilde{\mathrm{SL}(2, \R)} \times \R$ and $G =  \mathrm{Aff}^+(\R) \times \R$.

To begin with, let us instead  assume that $L= {\mathrm{SL}(2, \R)} \times \R$.
 The isotropy $I$ is not contained in the $\R$-factor, so it has a non-trivial projection $J$, a one-parameter group
of $\mathrm{SL}(2, \R)$. Since $\R \subset G$, we have $G. I = G. J = (\mathrm{Aff}^+(\R). J)\times \R$,  and thus the transitivity question is reduced to that
of $\mathrm{Aff}^+(\R)$ on $\mathrm{SL}(2, \R)/ J$, or equivalently the transitivity of $J$ on  $\mathrm{SL} (2, \R)/ \mathrm{Aff}^+(\R)$.  The last is the circle, endowed with the $\mathrm{SL}(2, \R)$-action by homography. Any hyperbolic or parabolic one-parameter group
acting by homography on the circle has (at least) a fixed point, and thus does not act transitively. In contrast, any elliptic one-parameter group has no fixed point and so has exactly one orbit.

Let us now come back to the general case $L= \widetilde{\mathrm{SL}(2, \R)} \times \R$.  Let $ p:   \widetilde{\mathrm{SL}(2, \R)} \to  {\mathrm{SL}(2, \R)}  $ be the projection, and
$\pi_1 \subset  \widetilde{\mathrm{SL}(2, \R)} $,   the fundamental group of $L=  {\mathrm{SL}(2, \R)}$.
It is cyclic and central.  Let $H$ be a connected subgroup of $L=  {\mathrm{SL}(2, \R)}$. In the case where $H$ is a parabolic of hyperbolic one-parameter group of $\mathrm{Aff}^+(\R)$,
the inverse image $p^{-1}(H)$ consists of countably many disjoint copies of $H$ embedded in $L= \widetilde{\mathrm{SL}(2, \R)} $. In other words,
the restricted  projection $p: p^{-1}(H) \to H$ is trivial.  In contrast, if $H$ is an elliptic one-parameter group, $p^{-1}(H)$ is connected,  contains $\pi_1$,  and is the universal cover
of $H \cong \SO(2)$.   These  details give us a full description of one-parameter subgroups of $  \widetilde{\mathrm{SL}(2, \R)}  $ and allows us to conclude
that transitivity holds only in the elliptic case.  Furthermore, $\pi_1$ is contained in the isotropy, and because it is central, it acts trivially on  the homogeneous space
$L/I$. In other words, there is  finally no need to consider the universal cover $\widetilde{\mathrm{SL}(2, \R)}$. 

\bigskip

\section{Further remarks}

 \subsection{Reductive complements}\label{Subsec:FR-red-comp} We recover the example of Sec. \ref{sec:sl2-aff-ellip}. Here, we have $\LL=\mathfrak{sl}(2,\R) \oplus \R$ with basis $\{X,Y,Z,T\}$ and $\g=\mathfrak{aff}(\R)\oplus \R$ with basis $\{X,Y,T\}$. We take as isotropy generator the element $U=-\tfrac{1}{2}Y-Z$. This corresponds in the notation of Sec. \ref{sec:sl2-aff-ellip} to having $\alpha =0$, $\beta=-\tfrac{1}{2}$ and $\gamma=-1$. So, indeed, we get elliptic isotropy and the metric is $m=\mathrm{Diag}(1,1,-1)$.

Now, $\LL=\mathrm{span}\{X,Y,T,U\}$. It can be computed that the reductive complements of $\I=\R U$ are given by the one-parameter family 
\begin{equation*}
    r_\alpha = \mathrm{span}\{X, Y+U, T +\alpha U\}, \, \, \alpha\in\R.
\end{equation*}
As can be seen, none of the $r_\alpha$ is a subalgebra since $[X,Y+U]=-U$ and $U\notin r_\alpha.$
Recall, as mentioned in Sec. \ref{subsec:isom-aut}, that it was proved in \cite{FelsRenner}, that when the isotropy is connected all semi-Riemannian homogeneous spaces are reductive. We have an example here of a transitive action of $\mathrm{Aff}^+(\R)\times \R$ on $L/I$, seen as the product $(\mathbb{H}^2, ds^2)\times (\R, -dt^2)$, such that none of the reductive complements of $\I$ is a subalgebra.  

\subsection{Sectional curvature and completeness}\label{Subsec:FR-sec-curv} 
Any semi-Riemannian manifold $(M,g)$ of dimension $n$ will have isometry group of maximal dimension $n(n+1)/2$ if and only if $g$ is complete and the sectional curvature is constant. In the setting under discussion in this article, of 3-dimensional metric Lie groups, it could be conjectured that dropping one of the two conditions,  completeness or constant sectional curvature, but keeping the other would be equivalent to a 4-dimensional Killing algebra. This is not the case.

Regarding completeness, recall that all metrics on $\mathfrak{euc}(2)$ are complete, \cite{BrombergMedina}, but none have 4-dimensional Killing algebra, cf. Table \ref{Table:Isot-1dim}. 
Concerning curvature, we encountered two metrics on $\mathfrak{aff}(\R)\oplus\R$ which have negative constant sectional curvature. One of them has 6-dimensional Killing algebra (see Sec. \ref{subsec:U-not-sl2}) and the other one has Killing algebra of dimension 4 (see Subsec. \ref{subsec:N=R2-ideal}). There is a third (non-equivalent, cf. Appendix \ref{Ap:normal-forms}) metric $p$ on $\mathfrak{aff}(\R)\oplus \R$, given in that standard basis  $\{X,Y,Z\}$ such that $[X,Y]=Y$, by
\begin{equation*}
  p(X,X) =1, \, p(Y,Z) =1, \, p(Z,Z) =1, 
\end{equation*}
which has constant negative sectional curvature and is incomplete.
From our discussion, this metric has Killing algebra of dimension 3 (since it does not appear in Table \ref{Table:Isot-1dim}).

The converse is, perhaps, of more interest. We now consider $\mathfrak{sol}$, with the standard basis $\{X,Y,Z\}$ such that $[X,Y]=Y, [X,Z]=-Z$, and the metric $q$ given by 
\begin{equation*}
    q(X,Z) =1, \, q(Y,Y)=1,
\end{equation*}
which is incomplete, \cite{BrombergMedina}, and not of constant sectional curvature. In Sec. \ref{sec:N=heis}, this metric was seen to be one of the hyperbolic plane waves of \cite[Th. 1.4]{AlloutBelkacemZeghib} and was proved to have 4-dimensional Killing algebra. The latter means that the Lie algebra of complete Killing fields is 4-dimensional. Moreover, there cannot exist other incomplete Killing fields as this would imply constant sectional curvature. This is an example of an incomplete metric whose Killing fields are all complete.

\appendix

\section{}\label{Ap:normal-forms}
\subsection{Normal forms of metric Lie algebras}

We present here the list of equivalence classes of metrics for each 3-dimensional Lie algebra $\g$ under the action of the automorphism group $\mathrm{Aut}(\g)$ and up to scaling, \cite{BoucettaChakkar-m, HaLee23}.

{\small
\begin{table}[H]
\hspace*{-7mm}	
\begin{tabular}{lll} \toprule
Lie algebra & Non-vanishing brackets \qquad& Metric normal forms \\ \toprule 
$\begin{array}{c} \mathbb{R}^3 \end{array}$ &\begin{tabular}{l} -- \end{tabular}&  $\begin{array}{c} (e^1)^2+(e^2)^2+\varepsilon(e^3)^2, \,\, \varepsilon =\pm 1 \end{array}$  \\ \midrule
$\begin{array}{c} \mathfrak{so}(3) \end{array}$ & \begin{tabular}{l}
     $[e_1, e_2] = e_3$  \\
     $[e_2, e_3] = e_1$ \\
      $ [e_3, e_1] = e_2$
\end{tabular} &  $\begin{array}{c} (e^1)^2+\alpha_1 (e^2)^2+\alpha_2 (e^3)^2, \,\, \alpha_1, \alpha_2 \neq 0 \end{array}$\\ \midrule
$\begin{array}{c} \mathfrak{sl}(2,\R) \end{array}$ & \begin{tabular}{l}
     $[e_1, e_2] = e_2$  \\
       $[e_3, e_1] = e_3$ \\
       $[e_2, e_3] = e_1$
\end{tabular} & $ \begin{array}{l}
(e^1)^2+\tfrac{1}{2}(\alpha_1+\alpha_2) ((e^2)^2+(e^3)^2)+ (\alpha_1-\alpha_2) (e^2e^3), \,\,  \alpha_1 \geq 1, \alpha_2 \neq 0 \smallskip\\
(e^1)^2+\tfrac{1}{2}(\alpha_1+\alpha_2) ((e^2)^2+(e^3)^2)+ (\alpha_1-\alpha_2) (e^2e^3), \,\,  \alpha_1, \alpha_2<0 \\
(e^1)^2+\alpha_1 ((e^2)^2-(e^3)^2)+ 2\alpha_2 (e^2e^3), \,\, \alpha_1 > 0 \\
(e^1)^2+2\alpha (e^2e^3) + \varepsilon (e^2)^2, \, \, \alpha \neq 0, \varepsilon = \pm 1   \\
(e^1)^2 + 2\sqrt{2} (e^1e^3) + 2 (e^2e^3)
\end{array}$ \\ \midrule 
$\begin{array}{c}
  \mathfrak{heis}  
     \end{array}$ & $\begin{array}{c}
          [e_1, e_2] = e_3 
     \end{array}$ & $\begin{array}{l} 
     (e^1)^2+(e^2)^2+\varepsilon (e^3)^2, \,\, \varepsilon =\pm 1\\
     (e^1)^2-(e^2)^2+(e^3)^2\\
     (e^1)^2+2(e^2e^3)
     \end{array}$ \\ \midrule
$\begin{array}{c}
    \mathfrak{sol}
\end{array}$ &  \begin{tabular}{l} $[e_1,e_2]=e_2$\\ $[e_3, e_1] =e_3$ \end{tabular} & $\begin{array}{l} 
     (e^1)^2+(e^2)^2+2(e^2e^3)+\alpha_1(e^3)^2, \,\, \alpha_1 \neq 1\\
     (e^1)^2-(e^2)^2+2(e^2e^3)+\alpha_2(e^3)^2, \,\, \alpha_2 \neq -1, \alpha_2 \leq 0\\
        (e^1)^2+2(e^2e^3)\\
       (e^1)^2+(e^2)^2+\varepsilon (e^3)^2,\,\, \varepsilon = \pm 1 \\
         (e^3)^2+(e^2)^2-(e^1)^2\\
        (e^2)^2+(e^3)^2+2(e^1e^3)+2(e^2e^3)\\
         (e^2)^2+2(e^1e^3)\\
          \end{array}$   \\ \midrule
$\begin{array}{l}
    \mathfrak{h}(1)
\end{array}$ &  \begin{tabular}{l} $[e_1,e_2]=e_2$\\ $[e_1, e_3] = e_3$ \end{tabular}& 
$\begin{array}{l} 
     (e^1)^2+(e^2)^2+\varepsilon(e^3)^2, \,\, \varepsilon = \pm 1\\
      (e^3)^2+(e^2)^2-(e^1)^2\\
     (e^3)^2+2(e^1e^2) \\
         \end{array}$ 
\\ \midrule
$\begin{array}{l}
    \mathfrak{h}(\lambda), \, |\lambda| <  1 \qquad
\end{array}$ &  \begin{tabular}{l} $[e_1,e_2]=e_2$\\ $[e_1, e_3] = \lambda e_3$ \end{tabular}& $\begin{array}{l} 
     (e^1)^2+(e^2)^2+2(e^2e^3)+\alpha_1(e^3)^2, \,\, \alpha_1 \neq 1\\
     (e^1)^2-(e^2)^2+2(e^2e^3)+\alpha_2(e^3)^2, \,\, \alpha_2 \neq -1\\
      (e^1)^2+\varepsilon (e^3)^2+2(e^2e^3), \,\, \varepsilon = \pm 1\\
        (e^1)^2+2(e^2e^3)\\
       (e^1)^2+(e^2)^2+\varepsilon (e^3)^2,\,\, \varepsilon = \pm 1 \\
       (e^1)^2-(e^2)^2+(e^3)^2 \\
        (e^3)^2+(e^2)^2-(e^1)^2\\
        (e^2)^2+(e^3)^2+2(e^1e^3)+2(e^2e^3)\\
         (e^2)^2+2(e^1e^3)\\
          (e^3)^2+2(e^1e^2)\\
     \end{array}$  \\
 \midrule
$\begin{array}{c}
     \mathfrak{psh}
\end{array}$ & \begin{tabular}{l}
     $[e_1,e_2] = e_2$  \\
     $[e_1, e_3] = e_2+ e_3 $
\end{tabular}& 
$\begin{array}{l} 
     (e^1)^2+(e^2)^2+\alpha_1(e^3)^2, \,\, \alpha_1 \neq 0\\
     (e^1)^2-(e^2)^2+\alpha_2(e^3)^2, \,\, \alpha_2 \neq 0\\    (e^1)^2+2\varepsilon(e^2e^3), \,\, \varepsilon = \pm 1\\
       (e^2)^2+2(e^1e^3) \\
        (e^3)^2+2(e^1e^2) 
     \end{array}$ 
\\ \midrule
$\begin{array}{l}\mathfrak{e}(\mu), \, \mu\geq 0 \end{array}$ & \begin{tabular}{l}
$[e_1, e_2] = \mu e_2+e_3$ \\ $[e_1,e_3] = \mu e_3-e_2$ 
\end{tabular} & $\begin{array}{l}
     (e^1)^2+(e^2)^2+\alpha_1(e^3)^2, \,\, 0< \alpha_1 \leq 1  \mbox{ or } \alpha_1 <0 \\
       (e^2)^2-(e^1)^2+\alpha_2(e^3)^2, \,\, 0<\alpha_2 \leq 1  \\
       (e^3)^2+2(e^1e^2)
\end{array}$\\
\bottomrule 
\end{tabular}
\vspace*{-4mm}
\caption{Normal forms of metric Lie algebras}   
\end{table}
}

\section*{Acknowledgments}  
{\small

\begingroup
\sloppy

The first and second named authors  acknowledge the support of CMAT (Centro de Matem\'atica da Universidade do Minho). Their research was financed by Portuguese Funds through FCT (Fundação para a Ciência e a Tecnologia, I.P.) within the Projects UIDB/00013/2020 and UIDP/00013/2020 and also within the doctoral grant UI/BD/154255/2022 of the first named author.   Partial support from LABEX MILYON (ANR-10-LABX-0070) of Université de Lyon, within the framework of the “France 2030” program (ANR-11-IDEX-0007), managed by the French National Research Agency (ANR) is acknowledged by the first named author. 

\endgroup
}

\bibliographystyle{alpha}
\bibliography{bibliography}

\end{document}